\documentclass[10pt,reqno]{amsart}
\usepackage{amsmath}
\usepackage{amsthm}
\usepackage{amssymb}
\usepackage{graphicx}
\usepackage{amsfonts}
\usepackage{latexsym}
\usepackage{setspace}
\usepackage{hyperref}
\usepackage{cite}
\usepackage{multicol}
\usepackage[active]{srcltx}

\newcommand{\B}{\mathcal{B}}
\newcommand{\C}{ \mathbb{C}}
\newcommand{\D}{ \mathbb{D}}
\newcommand{\dD}{ \partial\mathbb{D}}

\newcommand{\norm}[1]{\| #1 \|}

\newcommand{\A}{\mathcal{A}}

\newcommand{\W}{\mathcal{W}}
\newcommand{\T}{\mathcal{T}}
\newcommand{\h}{\mathcal{H}}
\newcommand{\K}{\mathcal{K}}

\renewcommand{\vec}[1]{{\bf #1}}

\renewcommand{\phi}{\varphi}

\numberwithin{equation}{section}
\theoremstyle{plain}
\newtheorem{Proposition}[equation]{Proposition}
\newtheorem{Corollary}[equation]{Corollary}
\newtheorem{Theorem}[equation]{Theorem}
\newtheorem{Lemma}[equation]{Lemma}
\theoremstyle{definition}

\newtheorem{Example}[equation]{Example}
\newtheorem{Remark}[equation]{Remark}

\allowdisplaybreaks

\begin{document}
\bibliographystyle{amsplain}

    \title{Spatial isomorphisms of algebras of truncated Toeplitz operators}

    \author{Stephan Ramon Garcia}
    \address{   Department of Mathematics\\
            Pomona College\\
            Claremont, California\\
            91711 \\ USA}
    \email{Stephan.Garcia@pomona.edu}
    \urladdr{http://pages.pomona.edu/\textasciitilde sg064747}

    \author{William T. Ross}
\address{   Department of Mathematics and Computer Science\\
            University of Richmond\\
            Richmond, Virginia\\
            23173 \\ USA}
    \email{wross@richmond.edu}
    \urladdr{http://facultystaff.richmond.edu/~wross}

\author[W.R.~Wogen]{Warren R. Wogen}
	\address{Department of Mathematics, University of North Carolina, Chapel Hill, North Carolina 27599}
	\email{wrw@email.unc.edu}
	\urladdr{http://www.math.unc.edu/Faculty/wrw/}

    \keywords{Toeplitz operator, model space, truncated Toeplitz operator, reproducing kernel, complex symmetric operator,  conjugation.}
    \subjclass[2000]{47A05, 47B35, 47B99}

    \thanks{First author partially supported by National Science Foundation Grant DMS-1001614.}

    \begin{abstract}
    	We examine when two maximal abelian algebras in the truncated Toeplitz operators are spatially isomorphic. 
This builds upon recent work of N.~Sedlock, who obtained a complete description of the maximal
	algebras of truncated Toeplitz operators.
    \end{abstract}

\maketitle

\section{Introduction}
	Let $H^2$ denote the Hardy space of the open unit disk $\D$, $H^{\infty}$ denote the bounded analytic functions on $\D$,
	and $L^{\infty}:= L^{\infty}(\dD)$, $L^2:= L^2(\dD)$ denote the
	usual Lebesgue spaces on the unit circle $\dD$ \cite{Duren, MR2261424}.
	To each non-constant inner function $\Theta$ we associate the \emph{model space} \cite{CR, N1, N2}
		$$\K_{\Theta} := H^2 \ominus \Theta H^2,$$
	which is a reproducing kernel Hilbert space corresponding to the kernel
	\begin{equation}\label{eq-ReproducingKernel}
		k_{\lambda}(z)
		:= \frac{1 - \overline{\Theta(\lambda)} \Theta(z)}{1 - \overline{\lambda} z}, \quad z,\lambda \in \D.
	\end{equation}
	We sometimes use the notation $k_{\lambda}^{\Theta}$ when we need to emphasize the dependence on the inner function $\Theta$.
	The model space $\K_{\Theta}$ carries the natural \emph{conjugation}
	\begin{equation}\label{eq-ModelConjugation}
		C f := \overline{ f z} \Theta,
	\end{equation}
	defined in terms of boundary functions \cite{G,G-P, G-P-II} and a computation shows that
	\begin{equation} \label{Conj-RK}
		[C k_{\lambda}](z) = \frac{\Theta(z) - \Theta(\lambda)}{z - \lambda}.
	\end{equation}	
	Since each kernel function \eqref{eq-ReproducingKernel} is bounded and since
	their span is dense in $\K_{\Theta}$, it follows that $\K_{\Theta} \cap H^{\infty}$ is dense in $\K_{\Theta}$.
	For each \emph{symbol} $\phi$ in $L^2$ the corresponding \emph{truncated Toeplitz operator}
	$A_{\phi}$ is the densely defined operator on $\K_{\Theta}$ given by the formula
	\begin{equation*}
		A_{\varphi} f := P_{\Theta}(\phi f), \quad f \in H^{\infty} \cap \mathcal{K}_{\Theta},
	\end{equation*}
	where $P_{\Theta}$ is the orthogonal projection of $L^2$ onto $\mathcal{K}_{\Theta}$.
	When we wish to be specific about the inner function $\Theta$, we write
	$A^{\Theta}_{\phi}$.  

	Interest in truncated Toeplitz operators has blossomed over the last few years
	\cite{BBK, BCFMT, MR2597679, TTOSIUES, MR2440673, NLEPHS, MR2468883, MR2418122, Sed},
	sparked by a series of illuminating observations and open problems provided by D.~Sarason \cite{Sarason}.
	Although one can pursue the subject of unbounded truncated Toeplitz operators
	much further \cite{MR2468883, MR2418122}, we focus here on those $A_{\phi}$ which 
	have a \emph{bounded} extension to $\K_{\Theta}$ and we denote this set by $\mathcal{T}_{\Theta}$.
	One can show that $\T_{\Theta}$ is weakly closed \cite[Thm.~4.2]{Sarason} and contains
	$A_{\phi}$  whenever $\phi \in L^{\infty}$. 
	On the other hand, every $A_{\phi} \in \T_{\Theta}$ can be represented by an unbounded symbol \cite[Thm.~3.1]{Sarason}. In fact, 
	\begin{equation} \label{Sarason-zero}
	A_{\phi_1}  = A_{\phi_2} \Leftrightarrow \phi_1 - \phi_2 \in \Theta H^2 + \overline{\Theta H^2}.
	\end{equation}
	Moreover, a recent preprint \cite{BCFMT} has revealed that there are bounded truncated Toeplitz operators $A_{\phi}$
	which cannot be represented by a bounded symbol.

	For a given pair of inner functions $\Theta_1$ and $\Theta_2$, Cima and the current authors recently obtained
	necessary and sufficient conditions for $\T_{\Theta_1}$ and $\T_{\Theta_2}$ to be
	\emph{spatially isomorphic} \cite{TTOSIUES}, meaning there exists
	 a unitary operator $U: \K_{\Theta_1} \to \K_{\Theta_2}$ such that $\T_{\Theta_1} = U^{*} \T_{\Theta_2} U$.
	We denote this relationship by $\T_{\Theta_1} \cong \T_{\Theta_2}$.
	In this paper we examine when certain \emph{algebras} of truncated Toeplitz operators are spatially isomorphic.

	Although $\T_{\Theta}$ is \emph{not} an algebra of operators (a simple counterexample can be deduced 
	from \cite[Thm.~5.1]{Sarason}),	 it does contain certain algebras of interest.  Two examples are 
	\begin{equation}\label{eq-Analytic}
		\{A_{\phi}: \phi \in H^{\infty}\},
	\end{equation}
	the set of \emph{analytic} truncated Toeplitz operators on $\K_{\Theta}$
	and
	\begin{equation} \label{eq-Co-Analytic}
	 \{A_{\overline{\phi}}: \phi \in H^{\infty}\},
	 \end{equation} 
	 the corresponding set of 
	\emph{co-analytic} truncated Toeplitz operators.  Algebras of the form \eqref{eq-Analytic}
	are of particular interest since a seminal result of D.~Sarason  \cite{Sarason-NF} states that \eqref{eq-Analytic} is precisely
	the commutant of the \emph{compressed shift} $A_z$ on $\mathcal{K}_{\Theta}$.

	Recently, N.~Sedlock \cite{Sed} determined all of the maximal abelian algebras in $\T_{\Theta}$. 
	These algebras $\B_{\Theta}^{a}$, where the parameter $a$  belongs to the extended complex plane 
	$\widehat{\C} := \C \cup \{\infty\}$, are described in detail in Section \ref{SectionSedlock}. 
	The purpose of this paper is to determine when two such \emph{Sedlock algebras} 
	are spatially isomorphic to each other.   In particular, we develop a precise condition 
	describing when $\B_{\Theta}^{a} \cong \B_{\Theta}^{a'}$. 
	For certain inner functions $\Theta$, there will be many $a \neq a'$ for which $\B_{\Theta}^{a} \cong \B_{\Theta}^{a'}$. 
	For others, it will be the case that $\B_{\Theta}^{a} \cong \B_{\Theta}^{a'}$ if and only if $a = a'$. 
 
	We also address the question as to whether or not the notion of spatial isomorphism can be replaced by the weaker notion of isometric isomorphism. 
	For example, given a finite Blaschke product $\Theta$ with distinct zeros, we will show that the algebras 
	$\B_{\Theta}^{a}$ and $\B_{\Theta}^{a'}$ are spatially isomorphic if and only if they are isometrically isomorphic. 
	As a consequence, we will show, for finite Blaschke products $\Theta_1, \Theta_2$, each with distinct zeros,  that the corresponding quotient algebras $H^{\infty}/\Theta_1 H^{\infty}$ and 
	$H^{\infty}/\Theta_2 H^{\infty}$ are isometrically isomorphic if and only if there is a unimodular constant 
	$\zeta$ and a disk automorphism $\psi$ such that $\Theta_1 = \zeta \Theta_2 \circ \psi$.
	
	An important reason to consider the problem of spatial isomorphisms of Sedlock algebras is that it gives us a useful tool to address the question: Which operators are unitarily equivalent to analytic truncated Toeplitz operators (which turn out to be the commutant of the compressed shift)? The authors in \cite{GPR} examine this question for matrices. Since the analytic truncated Toeplitz operators on some model space $\mathcal{K}_{\Theta}$ are the Sedlock algebra $\mathcal{B}^{0}_{\Theta}$, this naturally leads us to consider  spatial isomorphisms of Sedlock algebras. The results of this paper will show that if an operator $T$ is unitarily equivalent to an operator in some Sedlock algebra, with the parameter $a \not \in \dD$, then $T$ is unitarily equivalent to an analytic truncated Toeplitz operator.

\section{Sedlock algebras}\label{SectionSedlock}

	In \cite{Sed} N.~Sedlock examined the following subclasses of $\T_{\Theta}$.   For $a \in \C$, define 
	\begin{equation*}
		\B_{\Theta}^{a} := \left\{A_{\phi + a \overline{A_{z} C \phi} + c} \in \mathcal{T}_{\Theta}: \phi \in \K_{\Theta}, c \in \C\right\}.
	\end{equation*}
	The $C$ appearing in the previous line is the conjugation in \eqref{eq-ModelConjugation} on the model space $\K_{\Theta}$.
	Following Sedlock, one can extend the definition of $\B_{\Theta}^{a}$ to $a = \infty$ by adopting
	the convention that $\B_{\Theta}^{\infty}$ denotes the set of co-analytic truncated Toeplitz operators on $\K_{\Theta}$ from \eqref{eq-Co-Analytic}.
	
	In light of the fact that the map $\phi \mapsto \phi + a \overline{A_{z} C \phi}$ is \emph{linear},
	it follows immediately that each $\B_{\Theta}^a$ is a linear subspace of $\T_{\Theta}$.
	One of the main theorems of Sedlock's paper \cite{Sed} is that each $\B_{\Theta}^{a}$ is actually an abelian algebra. 
	We therefore refer to the algebras $\B_{\Theta}^a$ as \emph{Sedlock algebras}.

	Sedlock also observed that
	\begin{equation} \label{Sed-conj-flip}
		A \in \B_{\Theta}^a \quad \Leftrightarrow\quad
		 A^{*} \in \B_{\Theta}^{1/\overline{a}},
	\end{equation}
	whence the definition of $\B^{\infty}_{\Theta}$ 
	consistent with the fact that
	$\B^{0}_{\Theta} = \{A^{\Theta}_{\phi}: \phi \in H^{\infty}\}$ consists of the analytic truncated
	Toeplitz operators.  Indeed, we have $(\B^{0}_{\Theta})^{*} = \B^{\infty}_{\Theta}$. 

	Sedlock algebras can be described in several different, but equivalent, ways. 
	For each $a \in \D^{-} = \{|z| \leq 1\}$, one can consider the following rank-one perturbation of $A_{z}$ on $\mathcal{K}_{\Theta}$:
	\begin{equation}\label{eq-GeneralizedShift}
		S_{\Theta}^{a} := A_{z} + \frac{a}{1 - \overline{\Theta(0)} a} k_{0} \otimes C k_{0}.
	\end{equation}
	A result of Sarason shows that these rank-one perturbations of $A_{z}$ belong to $\T_{\Theta}$ \cite{Sarason}.
	In fact, for $a \in \dD$ one obtains the so-called Clark unitary operators \cite{CMR, MR0301534, MR2198367}.
	
	\begin{Remark} \label{Clark-remark}
	Let us take a moment to briefly describe some facts about these Clark operators $S_{\Theta}^{a}$, $a \in \dD$, since they will appear later on. See \cite{CMR, MR0301534, MR2198367} for more details. If $a \in \dD$, then 
	$$\Re\left(\frac{a + \Theta}{a - \Theta}\right)$$ is a positive harmonic function on $\D$ and so, by the Herglotz theorem \cite[p.~2]{Duren}, there is a positive  finite measure $\mu_{a}$ on $\dD$ with 
	$$\Re\left(\frac{a + \Theta(z)}{a - \Theta(z)}\right) = \int_{\dD} \frac{1 - |z|^2}{|\zeta - z|^2} d \mu_{a}(\zeta).$$
	The family of measures $\{\mu_a:  a \in \dD\}$ obtained in this way are called the Clark measures (sometimes called  Aleksandrov-Clark measures) for $\Theta$ and they turn out to be the spectral measures for $S_{\Theta}^a$, i.e., $S_{\Theta}^{a}$ is unitarily equivalent to the multiplication operator $g \mapsto \zeta g$ on $L^{2}(\mu_a)$.

	One can show that a carrier for $\mu_a$ is 
	$$E_{a} :=\left \{\zeta \in \dD: \lim_{r \to 1^{-}} \Theta(r \zeta) = a\right\},$$ 
	i.e., $\mu_{a}(\dD \setminus E_{a}) = 0$. 
	Since $\mu_a$ is carried by $E_{\alpha}$, a set of Lebesgue measure zero, it is singular with respect to Lebesgue measure.
	For example, if $\Theta$ is an $n$-fold Blaschke product, then $E_{a}$ is the set of $n$ (distinct)  points $\{\zeta_1, \zeta_2, \ldots, \zeta_n\} \subset \dD$ for which $\Theta(\zeta_j) = a$ and $\mu_a$ is given by 
	\begin{equation} \label{inner-atoms-intro}
	\mu_a = \sum_{j = 1}^{n} \frac{1}{|\Theta'(\zeta_j)|} \delta_{\zeta_j}.
	\end{equation}
	 If $\Theta$ is the atomic inner function 
	$$\Theta(z) = e^{-\frac{1 + z}{1 - z}},$$
	then, for each $a \in \dD$, $E_{a}$ is a countable set which clusters only at $\zeta = 1$. Moreover 
	$$\mu_a = \sum_{\Theta(\zeta) = a} \frac{|\zeta - 1|^2}{2} \delta_{\zeta}.$$
	\end{Remark}
	
	The following observation, essentially due to Sedlock \cite{Sed}, provides yet another description of $\B_{\Theta}^{a}$.
	
	\begin{Lemma} \label{alternate-description} 
	For each $a \in \widehat{\C}$ we have 
	\begin{equation} \label{Sed-right-form}
\mathcal{B}_{\Theta}^{a} = \{A_{\psi} \in \mathcal{T}_{\Theta}: \psi = \phi_0 (1 + a \overline{\Theta}) + c, \phi_0 \in \mathcal{K}_{\Theta}, \phi_{0}(0) = 0, c \in \C\}.
\end{equation}
	\end{Lemma}

\begin{proof}
		It is shown in \cite{Sed} that
		\begin{equation*} 
		\mathcal{B}_{\Theta}^{a} = \{A_{\psi} \in \mathcal{T}_{\Theta}: \psi = \phi_{0} + a\overline{A_{z} C \phi_{0}} + ck_0, \phi_0 \in \mathcal{K}_{\Theta}, \phi_{0}(0) = 0, c \in \C\}.
		\end{equation*}
		Since the function $\overline{\phi_{0}} \Theta$ belongs to $\K_{\Theta}$ (easily checked from the definition of $\mathcal{K}_{\Theta}$) it follows that
		\begin{equation*}
			A_{z} C \phi_{0} = P_{\Theta} (z \overline{z} \overline{\phi_{0}} \Theta) 
			= P_{\Theta}( \overline{\phi_0} \Theta) =  \overline{\phi_0} \Theta,
		\end{equation*}
		from which, using the fact that $A_{k_0} = I$, we get the desired conclusion. 
		\end{proof}

	Sedlock algebras can also be described succinctly in terms of commutants.
	Recall that for a collection $\A$ of bounded operators on a Hilbert space $\h$, the \emph{commutant} $\A'$ of $\A$
	is defined to be the set of all bounded operators on $\h$ which commute with every member of $\A$.

	\begin{Theorem}[Sedlock] \label{Sed-commutant}
		For any inner function $\Theta$ we have the following.
		\begin{enumerate}\addtolength{\itemsep}{0.5\baselineskip}
			\item For $a \in \D^{-}$, $\B_{\Theta}^{a} = \{S_{\Theta}^{a}\}'$.
			\item For $a \in \widehat{\C} \setminus \D^{-}$, $\B_{\Theta}^{a} = \{(S_{\Theta}^{1/\overline{a}})^{*}\}'$.
			\item If $a \neq a'$, then $\mathcal{B}^{a}_{\Theta} \cap \mathcal{B}^{a'}_{\Theta} = \C I$.
		\end{enumerate}
	\end{Theorem}
	
	As a consequence of Theorem \ref{Sed-commutant}, one sees that $\B_{\Theta}^a$, being the commutant of an operator,  is weakly closed.
	Sedlock goes on to show that each $\B_{\Theta}^a$ is a maximal algebra in $\T_{\Theta}$
	in the sense that every algebra in $\T_{\Theta} $ is contained in some Sedlock algebra $\B_{\Theta}^a$.
	We should also point out that Sedlock algebras are maximal in another natural sense. 
	Recall that an algebra $\A  \subset  \B(\h)$ is called \emph{maximal abelian} if 
	$\mathcal{A} = \mathcal{A}'$.  Since every algebra in $\T_{\Theta}$ is abelian \cite{Sed},
	it follows immediately from Theorem \ref{Sed-commutant} that every
	Sedlock algebra is maximal abelian.

It turns out that every member of a Sedlock algebra $\B_{\Theta}^a$ with $a \in \widehat{\C} \setminus \dD$ 
can be represented by a bounded symbol \cite{Sed}.  This is significant since
 there exists an inner function $\Theta$ and an $a \in \dD$ such that $\B_{\Theta}^a$ contains
a truncated Toeplitz operator which does not have a bounded symbol \cite{BCFMT}.

	Part (i) of Theorem \ref{Sed-commutant} asserts that the Sedlock algebra
	$\B_{\Theta}^{a}$, for $a \in \D^{-}$, is the commutant of $S_{\Theta}^{a}$.  However, we can say a bit more.
	For a bounded operator $A$ on a Hilbert space, we let $\W(A)$ denote the weak closure of $\{ p(A): \text{$p(z)$ a polynomial}\}$.
	In particular,  observe that $\W(A) \subseteq \{A\}'$.  
	
	\begin{Proposition} \label{BWS}
		For any inner function $\Theta$ we have the following. 
		\begin{enumerate}\addtolength{\itemsep}{0.5\baselineskip}
			\item If $a \in \D^{-}$, then $\B_{\Theta}^a = \W(S_{\Theta}^a)$.
			\item If $a \in \widehat{\C} \setminus \D^{-}$, then $\B_{\Theta}^a = \W((S_{\Theta}^{1/\overline{a}})^{*})$.
		\end{enumerate}
	\end{Proposition}
	
	The remainder of this section concerns Proposition \ref{BWS} and its proof. We state
	 a number of preliminary observations which will be useful later on.  
	Let us begin by observing that if $a \in \dD$, then
	$S_{\Theta}^a$ is a Clark unitary operator.  It is well-known, and discussed earlier in Remark \ref{Clark-remark},  that all such operators are cyclic 
	and possess a singular spectral measure on $\dD$ which is carried by the set $\{\Theta = a\}$.  
	Since $S_{\Theta}^a$ is cyclic, it follows from Fuglede's Theorem and the Double Commutant Theorem  that
	$\{S_{\Theta}^a\}'$ is the von Neumann algebra $\W^{*}(S_{\Theta}^a)$ generated by 
	$S_{\Theta}^a$ \cite{MR1721402}.  Since $S_{\Theta}^a$ is a singular unitary, an
	old result of J.~Wermer says that $\W(S_{\Theta}^a) = \W^{*}(S_{\Theta}^a)$ \cite[Thm.~6]{Wermer-normal}.  
	This establishes Proposition \ref{BWS} when $a \in \dD$.
	
	\begin{Remark} \label{Clark-needed}
	From the previous paragraph and from Remark \ref{Clark-remark}, we see that when $a \in \dD$, $\mathcal{B}_{\Theta}^{a}$ is spatially isomorphic to $L^{\infty}(\mu_a)$, where we think of $L^{\infty}(\mu_a)$ as the algebra of multiplication operators on $L^2(\mu_a)$ with symbols from $L^{\infty}(\mu_a)$. This was also observed by Sedlock \cite{Sed}.
	\end{Remark}
	
	To prove Proposition \ref{BWS} in the special case when $a = 0$, we require the following lemma
	which will itself prove useful later on.

	\begin{Lemma}\label{aiszero}
		For any inner function $\Theta$ we have $\mathcal{W}(S^{0}_{\Theta}) = \mathcal{B}^{0}_{\Theta}$.
	\end{Lemma}
 
	\begin{proof}
		Since $S_{\Theta}^0 = A_z$,
		it suffices to show, by \eqref{Sed-conj-flip}, that $\W(A_{\overline{z}}) = \B_{\Theta}^{\infty}$.
		Since the reverse inclusion $\supseteq$ is clear, we focus on establishing that 
		$\B_{\Theta}^{\infty} \subseteq \mathcal{W}(A_{\overline{z}})$.		
		For $g \in L^{\infty}$, we let $T_g$ denote the corresponding Toeplitz operator on 
		$H^2$ and recall that
		\begin{equation*}
			\mathcal{W}(T_{\overline{z}}) = \{T_g: \overline{g} \in H^{\infty}\} = \{T_{\overline{z}}\}'.
		\end{equation*}
		In light of the Commutant Lifting Theorem \cite{Sarason-NF}, it follows that
		\begin{equation*}
			\B_{\Theta}^{\infty} =\{A_{\overline{z}}\}' = \{T_{\overline{z}}\}'|\K_{\Theta} = \W(T_{\overline{z}})|\K_{\Theta}.
		\end{equation*}
		We now claim that $\W(T_{\overline{z}})|\K_{\Theta}$ is contained in $\W(A_{\overline{z}})$. 
		Indeed, if a sequence of polynomials $p_n(T_{\overline{z}})$ in $T_{\overline{z}}$ 
		converges weakly to $T_g$, then it follows that 
		$p_n(T_{\overline{z}})|\K_{\Theta} = p_n(A_{\overline{z}})$ converges weakly to $A_g$. 
		In particular, this demonstrates that
		 $\B_{\Theta}^{\infty} \subseteq \mathcal{W}(A_{\overline{z}})$ and concludes the proof.
	\end{proof}
	
	To complete the proof of Proposition \ref{BWS}, we require some additional notation.
	For $a \in \D$ we define
	\begin{align}
		b_{a}(z) &:= \frac{z - a}{1 - \overline{a} z}, \label{eq-Factor}\\[3pt]
		\Theta_a &:= b_{a} \circ \Theta. \nonumber
	\end{align}
	Now recall that for each $a \in \D$, the \emph{Crofoot transform}
	\begin{equation} \label{Cro-def}
		U_{a}: \mathcal{K}_{\Theta} \to \mathcal{K}_{\Theta_a}, \quad U_{a} f := \frac{\sqrt{1 - |a|^2}}{1 - \overline{a} \Theta} f
	\end{equation}
	is unitary \cite{Crofoot}  (see \cite[Sect.~13]{Sarason} for a thorough discussion of Crofoot transforms in the context
	of truncated Toeplitz operators).  Furthermore, it has the property that 
	\begin{equation} \label{Sarason-Crofoot}
		U_{a} S_{\Theta}^{a} U_{a}^{*} = S_{\Theta_a}^{0},
	\end{equation}
	where $S_{\Theta}^{a}$ is the generalization of the Clark operator defined in \eqref{eq-GeneralizedShift}.
	Using this observation, we see that 
	\begin{equation} \label{H/aH}
		\B_{\Theta}^a \cong\B_{\Theta_{a}}^{0} \quad \forall a \in \D.
	\end{equation}
In particular, the proof of Proposition \ref{BWS} for $a \in \D$ now follows from 
	Lemma \ref{aiszero}, \eqref{Sarason-Crofoot}, and \eqref{H/aH}. The proof in the case $a \in \widehat{\C} \setminus \D^{-}$ 
	is settled by appealing to \eqref{Sed-conj-flip}. 

	\begin{Remark} \label{normal-ops}
		When $a \in \dD$, the algebra $\B_{\Theta}^{a}$ is generated by a single unitary operator and is therefore an 
		algebra of normal operators.  The situation is quite different for $a \in \widehat{\C} \setminus \dD$.
		In \cite[Prop.~6.5]{TTOSIUES} it is shown that if $A$ belongs to $\B^{0}_{\Theta}$ and $A$ is normal, 
		then $A = c I$. Using \eqref{H/aH} one can see that the same is true for $\B^{a}_{\Theta}$ whenever $a \in \D$. 
		Although the same result still holds if $a \in \widehat{\C} \setminus \D^{-}$, to prove it one needs Proposition \ref{sharp-B} (see below) along with \eqref{H/aH}. 
	\end{Remark}

\section{Basic spatial isomorphisms}

 \subsection{The spatial isomorphisms $\Lambda_{a}$, $\Lambda_{\psi}$, and $\Lambda_{\#}$}\label{SubsectionBasic}
	It turns out that \emph{every} spatial isomorphism between Sedlock algebras can be written 
	as a product of certain fundamental spatial isomorphisms, which were used in 
	\cite[Thm.~3.3]{TTOSIUES} to determine when $\T_{\Theta_1} \cong \T_{\Theta_2}$
	holds for two inner functions $\Theta_1, \Theta_2$.  These spatial isomorphisms
	are explicitly defined in terms of unitary operators between $\K_{\Theta}$ spaces.

	The first basic building block is the Crofoot transform $U_{a}: \K_{\Theta} \to \K_{\Theta_{a}}$ 
	which we have already encountered in \eqref{Cro-def}.  Each Crofoot transform $U_a$ implements the following spatial isomorphism
	\cite[Prop.~4.2]{TTOSIUES}:
	\begin{equation} \label{La}
		\Lambda_{a}: \mathcal{T}_{\Theta} \to \mathcal{T}_{\Theta_a}, 
		\qquad \Lambda_a(A) := U_{a} A U_{a}^{*}.
	\end{equation}

	The second class of spatial isomorphisms arises from composition with a disk automorphism.
	To be more specific, for fixed disk automorphism $\psi$ we set
	\begin{equation*}
		U_{\psi}: \K_{\Theta} \to \K_{\Theta \circ \psi}, 
		\qquad U_{\psi} f := \sqrt{\psi'} (f \circ \psi).
	\end{equation*}
	A routine computation \cite[Prop.~4.1]{TTOSIUES} reveals that $U_{\psi}$ is unitary, 
	\begin{equation} \label{form-Jphi}
		U_{\psi} A_{\phi}^{\Theta} U_{\psi}^{*} = A^{\Theta \circ \psi}_{\phi \circ \psi},
	\end{equation}
	and
	\begin{equation*}
		U_{\psi} \T_{\Theta} U_{\psi}^{*} = \T_{\Theta \circ \psi}.
	\end{equation*}
	In particular, this implies that the map
	\begin{equation} \label{Lpsi}
		\Lambda_{\psi}: \T_{\Theta} \to \T_{\Theta \circ \psi}, \quad \Lambda_{\psi}(A) := U_{\psi} A U_{\psi}^{*}
	\end{equation}
	is a spatial isomorphism.  
	
	Our last class of spatial isomorphism arises from the unitary operator (discussed in \cite{TTOSIUES})
	\begin{equation*}
		U_{\#}: \K_{\Theta} \to \K_{\Theta^{\#}}, \quad [U_{\#} f](z) := \overline{C f(\overline{z})},
	\end{equation*}
	where $\Theta^{\#}(z) := \overline{\Theta(\overline{z})}$ and $C$ denotes the conjugation
	\eqref{eq-ModelConjugation} on $\K_{\Theta}$.  In terms of boundary functions on the unit circle
	$\dD$, this can be written as
	\begin{equation} \label{J-sharp-defn} 
		[U_{\#} f](z) = \overline{z} f(\overline{z}) \Theta^{\#}(z).
	\end{equation}
	Although the preceding does not appear to represent the boundary values of a function
	in $\K_{\Theta^{\#}}$, note that $f(\overline{z}) = \overline{f^{\#}(z)}$ whence
	$U_{\#} f$ is simply the conjugate, in the sense of \eqref{eq-ModelConjugation},
	of the function $f^{\#}$ in $\K_{\Theta^{\#}}$.
	A computation in \cite[Prop.~4.6]{TTOSIUES} now yields 
	\begin{equation} \label{form-Jsharp}
		U_{\#} A^{\Theta}_{\phi} U_{\#}^{*} = A^{\Theta^{\#}}_{\overline{\phi^{\#}}}
	\end{equation}
	and
	\begin{equation*}
		U_{\#} \T_{\Theta} U_{\#}^{*} = \T_{\Theta^{\#}},
	\end{equation*}
	giving us our final class of spatial isomorphisms
	\begin{equation} \label{Lsharp}
		\Lambda_{\#}: \T_{\Theta} \to \T_{\Theta^{\#}}, \qquad \Lambda_{\#}(A) := U_{\#} A U_{\#}^{*}.
	\end{equation}

\subsection{Images of Sedlock algebras}
	We now wish to discuss the images of the Sedlock algebras $\B_{\Theta}^a$ under the three basic
	spatial isomorphisms $\Lambda_{a}$, $\Lambda_{\psi}$, and $\Lambda_{\#}$ defined above.
	
	To this end, let us first note that the image of a maximal abelian algebra under a spatial isomorphism
	is also a maximal abelian algebra.  To be more specific, suppose that $\h_1$ and $\h_2$ are 
	Hilbert spaces, $\mathcal{A}_1$, $\mathcal{A}_{2}$  are linear subspaces of $\mathcal{B}(\h_1)$ and $\mathcal{B}(\h_2)$ respectively,  and that $\Lambda: \mathcal{A}_1 \to  \mathcal{A}_2$ is a spatial isomorphism, i.e., there is a unitary 
	$U: \mathcal{H}_1 \to \mathcal{H}_2$ such that 
	$\Lambda(A) = U A U^{*}$ for all $A \in \mathcal{A}_1$.
	If $\A$ is a maximal abelian algebra in $\mathcal{A}_1$, then its image $\Lambda(\A)$ is maximal abelian algebra in 
	$\mathcal{A}_2$.  In particular, any spatial isomorphism $\Lambda$ induces a bijection between the maximal 
	abelian algebras in $\mathcal{A}_1$ and those in $\mathcal{A}_2$.  In the setting of Sedlock algebras, we conclude that if 
	$\Lambda: \T_{\Theta_1} \to \T_{\Theta_2}$ is a spatial isomorphism, then 
	there is a bijection $g: \widehat{\C} \to \widehat{\C}$ such that 
	\begin{equation*}
		\Lambda(\B_{\Theta_1}^a) = \B_{\Theta_2}^{g(a)}.
	\end{equation*}
	The following three propositions explicitly describe the bijection $g$ for the basic
	classes of spatial isomorphisms which we introduced above.

	\begin{Proposition} \label{sharp-B}
		For any inner function $\Theta$ and $a \in \widehat{\C}$,
		\begin{equation} \label{sharp}
			\Lambda_{\#}( \mathcal{B}_{\Theta}^a)  = \mathcal{B}_{\Theta^{\#}}^{1/a}.
		\end{equation}
	\end{Proposition}

	\begin{proof}
		From \eqref{form-Jsharp}, the sharp operator $U_{\#}$ satisfies
		$U_{\#} A_{\phi}^{\Theta} U_{\#}^{*} = A^{\Theta^{\#}}_{\overline{\phi^{\#}}}, \phi \in L^2.$
		Thus for $\phi \in \mathcal{K}_{\Theta}$ with $\phi(0) = 0$ we have 
		\begin{align*}
		\Lambda_{\#}\left(A_{\phi (1 + a \overline{\Theta}) + c}\right) & = A_{\overline{(\phi (1 + a \overline{\Theta})+ c)^{\#}}}\\
		& = A_{\phi (\overline{z}) (1 + a \overline{\Theta(\overline{z})}) + c}\\
		& = A_{\phi (\overline{z}) \overline{\Theta(\overline{z})} (\Theta(\overline{z}) + a) + c}\\
		& = A_{\frac{1}{a} \phi (\overline{z})\overline{ \Theta(\overline{z}}) (1 + \frac{1}{a}\overline{\Theta^{\#}}) + c}.
		\end{align*}
		Note that since $\phi(0) = 0$, then $\phi(\overline{z}) \Theta^{\#} \in \mathcal{K}_{\Theta^{\#}}$.
		The result now follows from \eqref{Sed-right-form}.
	\end{proof}

	\begin{Proposition} \label{B-psi}
		For any inner function $\Theta$, disk automorphism $\psi$, and $a \in \widehat{\C}$ we have
		\begin{equation*}
			\Lambda_{\psi} (\mathcal{B}_{\Theta}^{a})  = \mathcal{B}_{\Theta \circ \psi}^{a}.
		\end{equation*}
	\end{Proposition}
	
	\begin{proof}
		Suppose that $A \in \mathcal{B}_{\Theta}^{a}$. By \eqref{Sed-right-form} 
		$$A = A_{\phi (1 + a \overline{\Theta}) + c}, \quad \phi \in K_{\Theta}, \phi(0) = 0, c \in \C.$$
		By \eqref{form-Jphi}, 
		$$\Lambda_{\psi}(A) = A^{\Theta \circ \psi}_{\phi \circ \psi (1 + a \overline{\Theta \circ \psi}) + c}.$$
		To show this operator belongs to $\mathcal{B}_{\Theta \circ \psi}^{a}$, we will use \eqref{Sed-right-form} and prove that there exists an $F \in \mathcal{K}_{\Theta \circ \psi}, F(0) = 0$, and a $d \in \C$ so that 
		\begin{equation} \label{Warren-trick}
		A^{\Theta \circ \psi}_{\phi \circ \psi (1 + a \overline{\Theta \circ \psi}) + c} = 
		A^{\Theta \circ \psi}_{F (1 + a \overline{\Theta \circ \psi}) + d}.
		\end{equation}
		
		To do this, let us first observe that if $P_{\Theta \circ \psi}$ is the orthogonal projection of $L^2$ onto $\mathcal{K}_{\Theta \circ \psi}$ and $P_{+}$ is the usual orthogonal projection of $L^2$ onto $H^2$, then 
		\begin{equation}\label{Ptcp}
		P_{\Theta \circ \psi} f = f - \Theta \circ \psi P_{+}(\overline{\Theta \circ \psi} f).
		\end{equation}
		Next we observe that by the conjugation $C$ from \eqref{eq-ModelConjugation} we know that $\overline{z \phi} \Theta \in \mathcal{K}_{\Theta} \subset H^2$. This means that $\phi \overline{\Theta} \in \overline{H^2}$ and so 
\begin{equation} \label{comp-C}
(\phi \circ \psi) \overline{\Theta \circ \psi} \in \overline{H^2}.
\end{equation}
Let us compute $P_{\Theta \circ \psi}(\phi \circ \psi)$:
\begin{align*}
P_{\Theta \circ \psi}(\phi \circ \psi) & = \phi \circ \psi - (\Theta \circ \psi)P_{+}(\phi \circ \psi \overline{\Theta \circ \psi}) \quad \mbox{(by \eqref{Ptcp})}\\
& =  \phi \circ \psi - (\Theta \circ \psi) (\phi \circ \psi)(0) \overline{(\Theta \circ \psi)(0)} \quad \mbox{(by \eqref{comp-C})}\\
& = (\phi \circ \psi  - (\phi \circ \psi)(0)) + (\phi \circ \psi)(0) (1 - (\Theta \circ \psi)\overline{(\Theta \circ \psi)(0)})\\
& = (\phi \circ \psi  - (\phi \circ \psi)(0)) + (\phi \circ \psi)(0) k_{0}^{\Theta \circ \psi}.
\end{align*}
Let 
$$F = \phi \circ \psi  - (\phi \circ \psi)(0)$$ and notice from the above calculation that 
\begin{equation} \label{AAA}
F \in \mathcal{K}_{\Theta \circ \psi}, F(0) = 0
\end{equation}
and 
\begin{equation} \label{BBBB}
P_{\Theta \circ \psi}(\phi \circ \psi) = F + (\phi \circ \psi)(0) k_{0}^{\Theta \circ \psi}.
\end{equation}
A similar computation will show that 
\begin{equation} \label{CCC}
P_{\Theta \circ \psi}((\overline{\phi \circ \psi})( \Theta \circ \psi)) = (\Theta \circ \psi) \overline{F}.
\end{equation}

Since $\phi \circ \psi$ and $(\overline{\phi \circ \psi})(\Theta \circ \psi) \in H^2$ (see \eqref{comp-C}) we know, from basic properties of projections, that 
\begin{equation} \label{DDD}
\phi \circ \psi - P_{\Theta \circ \psi}(\phi \circ \psi) \in (\Theta \circ \psi)H^2
\end{equation} 
\begin{equation} \label{EEE}
 (\overline{\phi \circ \psi})(\Theta \circ \psi) - P_{\Theta \circ \psi}((\overline{\phi \circ \psi})(\Theta \circ \psi)) \in (\Theta \circ \psi) H^2.
\end{equation}

By \eqref{BBBB} and \eqref{DDD}, along with the identity $A_{k_0} = I$,
\begin{equation} \label{FFF}
A^{\Theta \circ \psi}_{\phi \circ \psi} = A^{\Theta \circ \psi}_{F + (\phi \circ \psi)(0) k_{0}^{\Theta \circ \psi}} = A^{\Theta \circ \psi}_{F + (\phi \circ \psi)(0)}.
\end{equation}
By \eqref{CCC} and \eqref{EEE}
$$A^{\Theta \circ \psi}_{\overline{\phi \circ \psi} (\Theta \circ \psi)} = A^{\Theta \circ \psi}_{(\Theta \circ \psi) \overline{F}}.$$
Now take adjoints on both sides of the above equation to get 
\begin{equation} \label{GGG}
A^{\Theta \circ \psi}_{\phi \circ \psi (\overline{\Theta \circ \psi})} = A^{\Theta \circ \psi}_{F (\overline{\Theta \circ \psi})}.
\end{equation}
Combine \eqref{FFF} and \eqref{GGG} to obtain 
$$A^{\Theta \circ \psi}_{\phi \circ \psi + a (\phi \circ \psi) \overline{\Theta \circ \psi}} = A^{\Theta \circ \psi}_{F + a F (\overline{\Theta \circ \psi}) + (\phi \circ \psi)(0)}.$$
By \eqref{AAA} we have verified \eqref{Warren-trick} and thus the proof is complete. 
	\end{proof}

	\begin{Proposition} \label{delayed}
		For any inner function $\Theta$, $c \in \D$, and $a \in \widehat{\C}$, we have
		\begin{equation*}
			\Lambda_{c}( \mathcal{B}_{\Theta}^{a})  = \mathcal{B}_{\Theta_{c}}^{\ell_{c}(a)},
		\end{equation*}
		where 
		\begin{equation}\label{eq-lca}
			\ell_{c}(a) :=
			\begin{cases}
				\dfrac{a - c}{1 - \overline{c} a} & \text{if $a \neq \dfrac{1}{\,\overline{c}\,}$},\\[10pt]
				\infty & \text{if $a = \dfrac{1}{\,\overline{c}\,}$}.
			\end{cases}
		\end{equation}
	\end{Proposition}

	\begin{proof}
		Let us first  show that
		\begin{equation}\label{eq-Reinterpret}
			\Lambda_{c}(S^{a}_{\Theta}) = S_{\Theta_c}^{\ell_{c}(a)}, \quad a \in \D^{-}, c \in \D.
		\end{equation}
		To this end, we appeal to \cite[Lemma 13.2]{Sarason} to obtain the identities 
		\begin{equation*}
			U_{c} k_{0}^{\Theta} = \frac{1 - c \overline{\Theta(0)}}{\sqrt{1 - |c|^2}} k_{0}^{\Theta_c}, 
			\qquad U_{c} (C_{\Theta} k_{0}^{\Theta}) = \frac{1 - \overline{c} \Theta(0)}{\sqrt{1 - |c|^2}} C_{\Theta_c} k_{0}^{\Theta_c},
		\end{equation*}
		where $k_0^{\Theta}$ and $C_{\Theta} k_0^{\Theta}$ are defined by
		\eqref{eq-ReproducingKernel} and \eqref{Conj-RK}, respectively \footnote{Note that we need a subscript $\Theta$ on $C$ in order to distinguish the conjugation on $\mathcal{K}_{\Theta}$ from the conjugation on $\mathcal{K}_{\Theta_c}$.}. Therefore
		\begin{align*}
			\Lambda_{c}(k_{0}^{\Theta} \otimes C_{\Theta} k_{0}^{\Theta}) 
			& =   \left(\frac{1 - c \overline{\Theta(0)}}{\sqrt{1 - |c|^2}} k_{0}^{\Theta_c}\right)
			 \otimes \left(\frac{1 - \overline{c} \Theta(0)}{\sqrt{1 - |c|^2}} C_{\Theta_c} k_{0}^{\Theta_c}\right)\\
			& =  \frac{(1 - \overline{c} \Theta(0))^2}{1 - |c|^2} k_{0}^{\Theta_c} \otimes C_{\Theta_c}  k_{0}^{\Theta_c}.
		\end{align*}
	    Recall that \cite[Lemma 13.3]{Sarason} asserts 
		that $\Lambda_{c}(S_{\Theta}^c) = S_{\Theta_c}^{0}$. In light of the fact that
		\begin{align*}
			S_{\Theta}^{a} 
			& =  S_{\Theta}^{c} + \left( \frac{a}{1 - a \overline{\Theta(0)}} - \frac{c}{1 - c \overline{\Theta(0)}}\right) 
				k_{0}^{\Theta} \otimes C_{\Theta} k_{0}^{\Theta}\\
			& =  S_{\Theta}^{c} + \frac{a - c}{(1 - a \overline{\Theta(0)})(1 - c \overline{\Theta(0)})}k_{0}^{\Theta} \otimes C_{\Theta} k_{0}^{\Theta},
		\end{align*}
		we conclude that
		\begin{align*}
			\Lambda_{c}(S_{\Theta}^{a}) 
			& =  \Lambda_{c}\left(S_{\Theta}^{c} + \frac{a - c}{(1 - a \overline{\Theta(0)})(1 - c \overline{\Theta(0)})}k_{0}^{\Theta} 
				\otimes C_{\Theta} k_{0}^{\Theta}\right)\\
			& =  S_{\Theta_c}^{0} +  \frac{a - c}{(1 - a \overline{\Theta(0)})(1 - c \overline{\Theta(0)})} 
				 \frac{(1 - \overline{c} \Theta(0))^2}{1 - |c|^2} k_{0}^{\Theta_c} \otimes C_{\Theta_c} k_{0}^{\Theta_c}\\
			&=  S_{\Theta_c}^{0} + \frac{(a - c)(1 - c \overline{\Theta(0)})}{(1 - |c|^2)(1 - a \overline{\Theta(0)})} 
				k_{0}^{\Theta_c} \otimes C_{\Theta_c} k_{0}^{\Theta_c}.
		\end{align*}
		Recalling the definition \eqref{eq-GeneralizedShift}, we see that it suffices to demonstrate that
		\begin{equation*}
			 \frac{(a - c)(1 - c \overline{\Theta(0)})}{(1 - |c|^2)(1 - a \overline{\Theta(0)})} 
			 = \frac{\ell_{c}(a)}{1 - \ell_{c}(a) \overline{\Theta_c(0)}}.
		\end{equation*}
		However, the right-hand side of the preceding can be written as
		\begin{equation*}
			\frac{(a - c)(1 - c \overline{\Theta(0)})}{(1 - \overline{c} a)(1 - c \overline{\Theta(0)}) - (a - c) (\overline{\Theta(0)} - \overline{c})}
			= \frac{(a - c)(1 - c \overline{\Theta(0)})}{(1 - |c|^2)(1 - a \overline{\Theta(0)})}.
		\end{equation*}
		This proves \eqref{eq-Reinterpret}. Using Proposition \ref{BWS}, this also proves the proposition in the case $a \in \D^{-}$.

		Suppose that $a \in \widehat{\C} \setminus \D^{-}$ and recall from \eqref{Sed-conj-flip} that $\B^{a}_{\Theta} = (\B_{\Theta}^{1/\overline{a}})^{*}$.
		By \eqref{eq-Reinterpret}, it follows that 
		\begin{equation*}
			\Lambda_{c}(\mathcal{B}_{\Theta}^{1/\overline{a}}) = \mathcal{B}_{\Theta_c}^{\ell_{c}(1/\overline{a})},
		\end{equation*}
		whence, by the definition of $\ell_c(a)$ from \eqref{eq-lca}, we conclude that
		\begin{equation*}
			\Lambda_{c}(\mathcal{B}_{\Theta}^{a}) = \mathcal{B}_{\Theta_c}^{1/\overline{\ell_c(\frac{1}{\overline{a}})}}
			= \B_{\Theta_c}^{\ell_c(a)}.\qedhere
		\end{equation*}
	\end{proof}

\subsection{Words of unitary operators}
	Composing any of the basic spatial isomorphisms $\Lambda_{a}$, $\Lambda_{\psi}$, and $\Lambda_{\#}$
	introduced in Subsection \ref{SubsectionBasic} naturally leads one to consider words in the corresponding
	unitary operators $U_a$, $U_{\psi}$, and $U_{\#}$ and their adjoints.  The following proposition
	lists many of the basic words that arise in our work.

	\begin{Proposition} \label{J-rel}
		If $\Theta$ is an inner function, then
		\begin{multicols}{2}
			\begin{enumerate}\addtolength{\itemsep}{0.5\baselineskip}
				\item ${U_{b} U_{a} = \frac{|1 + \overline{b} a |}{1 + \overline{b} a} U_{\frac{a + b}{1 + b \overline{a}}}}$ \label{a}
				\item $U_{a}^{*} = U_{-a}$
				\item $U_{\phi} U_{\psi} = U_{\psi \circ \phi}$
				\item $U_{\phi}^{*} = U_{\phi^{-1}}$ \columnbreak
				\item $U_{\psi} U_b = U_b U_{\psi}$
				\item $U_{\#} U_a = U_{\overline{a}} U_{\#}$
				\item $U_{\#} U_{\psi} = U_{\psi^{\#}} U_{\#}$
			\end{enumerate}
		\end{multicols}
	\end{Proposition}

	\begin{proof}[Proof of (i) and (ii)]
		To obtain (i), we employ the identity
		\begin{equation*}
		1 - \left|\frac{\overline{a} + \overline{b}}{1 + \overline{b} a}\right|^2 = \frac{(1 - |a|^2) (1 - |b|^2)}{|1 + \overline{b} a|^2},
		\end{equation*}
		from which it follows that
		\begin{align*}
			U_b U_a f 
			& = U_{b}\left(\frac{\sqrt{1 - |a|^2}}{1 - \overline{a} \Theta} f\right)\\
			& = \frac{\sqrt{1 - |b|^2}}{1 - \overline{b} \Theta_a} \frac{\sqrt{1 - |a|^2}}{1 - \overline{a} \Theta} f\\
			& = \frac{\sqrt{1 - |b|^2}}{1 - \overline{b}( \frac{\Theta - a}{1 - \overline{a} \Theta})} \frac{\sqrt{1 - |a|^2}}{1 - \overline{a} \Theta} f\\
			& = \frac{\sqrt{1 - |a|^2}\sqrt{1 - |b|^2}  }{1 - \overline{a}\Theta - \overline{b}\Theta + a \overline{b}}f \\
			& = \frac{\sqrt{1 - |b|^2} \sqrt{1 - |a|^2}}{1 + \overline{b} a}\cdot \frac{f}{1 - \frac{\overline{a} + \overline{b}}{1 + \overline{b} a} \Theta} \\
			& = \frac{\sqrt{1 - |b|^2} \sqrt{1 - |a|^2}}{\sqrt{1 -\left |\frac{\overline{a} + \overline{b}}{1 + \overline{b} a}\right|^2}}
				\cdot \frac{\sqrt{1 - \left|\frac{\overline{a} + \overline{b}}{1 + \overline{b} a}\right|^2}}{1 + \overline{b}a}
				\cdot \frac{1}{1 - \frac{\overline{a} + \overline{b}}{1 + \overline{b} a}\Theta} f\\
			&= \frac{|1 + \overline{b} a|}{1 + \overline{b} a}\cdot 
				\frac{\sqrt{1 - \left|\frac{\overline{a} + \overline{b}}{1 + \overline{b} a}\right|^2}}
				{1 - \frac{\overline{a} + \overline{b}}{1 + \overline{b} a}\Theta} f \\
			&=\frac{|1 + \overline{b} a |}{1 + \overline{b} a} U_{\frac{a + b}{1 +b  \overline{ a}}} f .
		\end{align*}
		Statement (ii) follows immediately from (i) and the definition \eqref{Cro-def} of the Crofoot transform $U_a$.
	\end{proof}
	
	\begin{proof}[Proof of (iii) and (iv)]
		For (iii), simply note that
		\begin{align*}
			U_{\phi} U_{\psi} f 
			& = U_{\phi} \sqrt{\psi'} (f \circ \psi)\\
			& = \sqrt{\phi'} \sqrt{\psi'(\phi)} f(\psi(\phi))\\
			& = \sqrt{(\psi \circ \phi)'} f \circ (\psi \circ \phi)\\
			& = U_{\psi \circ \phi} f.
		\end{align*}
		Statement (iv) is an immediate consequence of (iii).
	\end{proof}

	\begin{proof}[Proof of (v)]
		This is a straightforward computation:
		\begin{align*}
			U_{\psi} U_b f & = U_{\psi} \left( \frac{\sqrt{1 - |b|^2}}{1 - \overline{b} \Theta} f\right)\\
			& =  \sqrt{\psi'} \frac{\sqrt{1 - |b|^2}}{1 - \overline{b} (\Theta \circ \psi)} (f \circ \psi)\\
			& = U_{b} U_{\psi} f. \qedhere
		\end{align*}
	\end{proof}

	\begin{proof}[Proof of (vi)]
		Regarding $z$ as an element of the unit circle, we use \eqref{J-sharp-defn} to obtain
		\begin{align*}
			U_{\#} U_{a} f  
			& = U_{\#} \left( \frac{\sqrt{1 - |a|^2}}{1 - \overline{a} \Theta} f\right)\\
			& = \frac{\sqrt{1 - |a|^2}}{1 - \overline{a} \Theta(\overline{z})} \overline{z} f(\overline{z}) (\Theta_a)^{\#}\\
			& =  \frac{\sqrt{1 - |a|^2}}{1 - \overline{a} \Theta(\overline{z})} \overline{z} f(\overline{z})
				 \frac{\overline{\Theta(\overline{z})} - \overline{a}}{1 - a \overline{\Theta(\overline{z})}}\\
			& = \frac{\sqrt{1 - |a|^2}}{1 - \overline{a} \Theta(\overline{z})} \overline{z} f(\overline{z}) 
				\frac{ \Theta^{\#}(z)(1 - \overline{a} \Theta(\overline{z}))}{1 - a \Theta^{\#}(z)}\\
			& = \frac{\sqrt{1 - |a|^2}}{1 - a \Theta^{\#}} \overline{z} f(\overline{z}) \Theta^{\#}\\
			& = U_{\overline{a}} U_{\#} f. \qedhere
		\end{align*}
	\end{proof}
	
	\begin{proof}[Proof of (vii)]
		We first note that for any disk automorphism 
		\begin{equation*}
			\psi(z) = \zeta \frac{z - c}{1 - \overline{c} z},  \tag{$\zeta \in \dD, c \in \D$}
		\end{equation*}
		a simple computation shows that
		\begin{equation} \label{ppp}
			\sqrt{\psi'(\overline{z})} \overline{z} = \sqrt{(\psi^{\#})^{'}} \psi(\overline{z}), \quad z \in \dD.
		\end{equation}
		Using \eqref{J-sharp-defn} we conclude that
		\begin{align*}
			U_{\#} U_{\psi} f 
			& = U_{\#} \sqrt{\psi'} (f \circ \psi)\\
			& = \sqrt{\psi'(\overline{z})} (f \circ \psi)(\overline{z}) \overline{z} (\Theta \circ \psi)^{\#}\\
			& = \sqrt{\psi'(\overline{z})} (f \circ \psi)(\overline{z}) \overline{z} \overline{\Theta(\overline{\psi})}\\
			& =  \sqrt{(\psi^{\#})^{'}} \psi(\overline{z}) f(\psi(\overline{z})) \overline{z} \overline{\Theta(\overline{\psi})} && \text{(by \eqref{ppp})}\\
			& =  \sqrt{(\psi^{\#})^{'}} \overline{\psi^{\#}(z)} f(\overline{\psi^{\#}(z)}) \Theta^{\#} \circ \psi^{\#}\\
			& = U_{\psi^{\#}} (\overline{z} f(\overline{z}) \Theta^{\#})\\
			& = U_{\psi^{\#}} U_{\#} f. \qedhere
		\end{align*}
	\end{proof}

	Maintaining the notation \eqref{La}, \eqref{Lpsi}, and \eqref{Lsharp} established in
	Subsection \ref{SubsectionBasic}, we see that Proposition \ref{J-rel} has the following
	immediate corollary.

	\begin{Corollary} \label{formulas-L}
		\quad\newline\vspace{-0.25in}
		\begin{multicols}{2}
			\begin{enumerate}\addtolength{\itemsep}{0.5\baselineskip}
				\item $\Lambda_{b} \Lambda_{a} =  \Lambda_{\frac{a + b}{1 + b \overline{a}}}$
				\item $\Lambda_{a}^{-1} = \Lambda_{-a}$
				\item $\Lambda_{\phi} \Lambda_{\psi} = \Lambda_{\psi \circ \phi}$
				\item $\Lambda_{\phi}^{-1} = \Lambda_{\phi^{-1}}$ \columnbreak
				\item $\Lambda_{\psi} \Lambda_b = \Lambda_b \Lambda_{\psi}$.
				\item $\Lambda_{\#} \Lambda_a = \Lambda_{\overline{a}} \Lambda_{\#}$
				\item $\Lambda_{\#} \Lambda_{\psi} = \Lambda_{\psi^{\#}} \Lambda_{\#}$
			\end{enumerate}
		\end{multicols}
		\vspace{-0.1in}
		Consequently, any finite word in the $\Lambda$ spatial isomorphisms as above can be written as 
		$\Lambda =  \Lambda_{a} \Lambda_{\psi}$ or 
		$\Lambda = \Lambda_{a} \Lambda_{\#} \Lambda_{\psi}$, where we allow $a = 0$ and $\psi(z) = z$.
	\end{Corollary}

\subsection{Spatial isomorphisms of $\T_{\Theta}$ spaces}
%
	
	In \cite{TTOSIUES}, Cima and the authors
	showed that for two inner functions $\Theta_1$ and $\Theta_2$ the corresponding 
	spaces $\T_{\Theta_1}$ and $\T_{\Theta_2}$ of truncated Toeplitz operators are spatially isomorphic, i.e., $\mathcal{T}_{\Theta_1} \cong \mathcal{T}_{\Theta_2}$, if and only if either 
	\begin{equation*}
	\Theta_1 = \phi \circ \Theta_2 \circ \psi \quad \mbox{or} \quad \Theta_1 = \phi \circ (\Theta_2)^{\#} \circ \psi
	\end{equation*}
	for some disk automorphisms $\phi$ and $\psi$. Informally speaking, 
	the $\psi$ will come from applying the $\Lambda_{\psi}$ spatial isomorphism \eqref{Lpsi}, the $\Theta^{\#}$
	from applying $\Lambda_{\#}$ \eqref{Lsharp}, and $\phi$ from applying $\Lambda_a$ \eqref{La}. 
	We make this more precise with the following theorem.

	\begin{Theorem} \label{most-three-SI}
		If $\Lambda: \T_{\Theta_1} \to\T_{\Theta_2}$ is a spatial isomorphism, then
		$\Lambda =  \Lambda_{a} \Lambda_{\psi}$ or $\Lambda = \Lambda_{a} \Lambda_{\#} \Lambda_{\psi}$, 
		where we allow $a = 0$ and $\psi(z) = z$.
	\end{Theorem}
	
	\begin{proof}
		The proof of \cite[Thm.~3.3]{TTOSIUES} shows that there exists an inner function $\Theta$ and
		a finite sequence $\Lambda_1, \Lambda_2,\ldots, \Lambda_n$ of spatial isomorphisms from among the 
		families $\Lambda_{\psi}$, $\Lambda_{\#}$, and $\Lambda_a$ so that 
		\begin{equation*}
			(\Lambda_1 \cdots \Lambda_s) \Lambda (\Lambda_{s + 1} \cdots \Lambda_{n})
		\end{equation*}
		is the identity on $\T_{\Theta}$. Now apply Corollary \ref{formulas-L}.
	\end{proof}

\subsection{A density detail} In the next section we will need the following density result. We would like to thank Roman Bessonov for pointing this out to us.

\begin{Proposition} \label{density-Bess}
For any inner function $u$, the set $\{A_{\phi}^{u}: \phi \in L^{\infty}\}$ is weakly dense in $\mathcal{T}_{u}$. 
\end{Proposition}

\begin{proof}
In \cite{BBK} they define the space 
$$X_{u} := \left\{\sum f_j \overline{g_j}: f_j, g_j \in \mathcal{K}_u, \sum_{j} \|f_j\| \|g_j\| < \infty\right\}$$ with norm defined as the infimum of $\sum \|f_j\| \|g_j\|$ over all possible representations of the element of the form $\sum f_j \overline{g_j}$. Notice, by the Cauchy-Schwarz inequality, that $\sum f_j \overline{g_j}$ converges in $L^1$ and so $X_u \subset L^1$. In the same paper they show that the dual of $X_u$ can be isometrically identified with $\mathcal{T}_u$ via the pairing 
$$\left(\sum f_j \overline{g_j}, A\right) := \sum \langle A f_j, g_j \rangle.$$ They go on further to show that the ultra-weak topology on $\mathcal{T}_u$, given by the above pairing, coincides with the weak topology on $\mathcal{T}_u$. 

So to show that $\{A_{\phi}^u: \phi \in L^{\infty}\}$ is weakly dense in $\mathcal{T}_u$, we just need to show that the pre-annihilator of this set is zero. To this end, suppose $F = \sum f_j \overline{g_j} \in X_u$ with 
$(F, A_{\phi}) = 0$ for all $\phi \in L^{\infty}$. Using the fact that $\phi$ is bounded and the sum defining $F$ converges in $L^1$ we see that 
$$(F, A^u_{\phi}) = \sum \langle A^u_{\phi} f_j, g_j \rangle  = \sum \int \phi f_i \overline{g_j} dm = \int \phi \sum f_j \overline{g_j} dm = \int \phi F dm$$
for all $\phi \in L^{\infty}$. Since $F \in L^1$, we conclude that $F = 0$ almost everywhere and so the pre-annihilator of 
$\{A^{u}_{\phi}: \phi \in L^{\infty}\}$ is zero. 
\end{proof}

\begin{Remark}
It can be the case, for example when $u$ is a one-component inner function \cite{BBK},  that $\{A_{\phi}^{u}: \phi \in L^{\infty}\} = \mathcal{T}_u$, i.e., every bounded truncated Toeplitz operator on $\mathcal{K}_u$ has a bounded symbol. It can also be the case that $\{A_{\phi}^{u}: \phi \in L^{\infty}\}$ is a proper subset of $\mathcal{T}_u$ \cite{BCFMT}. In either case, Proposition \ref{density-Bess} shows that $\{A_{\phi}^{u}: \phi \in L^{\infty}\} $ is weakly dense in $\mathcal{T}_u$. 
\end{Remark}

\section{Spatial isomorphisms of Sedlock algebras}

	For a fixed inner function $\Theta$ and $a, a' \in \widehat{\C}$, when is $\mathcal{B}^{a}_{\Theta} \simeq \mathcal{B}^{a'}_{\Theta}$? When $a, a' \in \dD$ it is possible to give a complete answer. 	
	For a positive measure $\mu$ on $\dD$, let
	$\kappa(\mu) = (\epsilon, n)$ where $0 \leq n \leq \infty$ is the number of atoms of $\mu$ and $\epsilon$ is 
	$0$ if $\mu$ is purely atomic and $1$ if $\mu$ has a (non-zero) continuous part. 
	An old theorem of Halmos and von Neumann \cite{MR1721402, MR0006617} asserts that
	$L^{\infty}(\mu) \cong L^{\infty}(\nu)$ (considered as multiplication operators on $L^2(\mu)$, respectively $L^2(\nu)$) if and only if $\kappa(\mu) = \kappa(\nu)$.

	\begin{Theorem} \label{Clark}
		If $\Theta$ is an inner function, $a, a' \in \dD$, and $\mu_a, \mu_{a'}$ denote the corresponding
		Clark measures, then
		\begin{equation*}
			\B_{\Theta}^a \cong \B_{\Theta}^{a'} 
			\quad\Leftrightarrow \quad
			\kappa(\mu_a) = \kappa(\mu_{a'}).
		\end{equation*}
	\end{Theorem}

	\begin{proof}
		From our discussion in Remark \ref{normal-ops} we have the spatial isomorphisms  $\B_{\Theta}^{a} \cong L^{\infty}(\mu_a)$ and 
		$\B_{\Theta}^{a'} \cong L^{\infty}(\mu_{a'})$. Applying the Halmos-von Neumann theorem referred to above
		yields the result.
	\end{proof}
	
	\begin{Corollary} \label{nothingtoprove}
		If $\Theta$ is a finite Blaschke product, then $\B_{\Theta}^{a} \cong \B_{\Theta}^{a'}$ whenever $a, a' \in \dD$.
	\end{Corollary}

	\begin{proof}
		Let $n$ denote the number of zeros of $\Theta$, counted according to their multiplicity.
		If $a, a' \in \dD$, then, from \eqref{inner-atoms-intro},  the Clark measures $\mu_a$ and $\mu_{a'}$ are both discrete and each consists  
		precisely of $n$ atoms (see also \cite[p.~207]{CMR}).
	\end{proof}

	For a finite Blaschke product $\Theta$, the preceding corollary indicates that the Sedlock algebras
	$\B_{\Theta}^a$ for $a \in \dD$ are all mutually spatially isomorphic.  In other words, spatial isomorphism
	induces an equivalence relation upon these algebras which yields precisely one equivalence class.  It is somewhat
	surprising, however, to learn that there exists an inner function $\Theta$ for which the Sedlock algebras
	$\B_{\Theta}^a$ for $a \in \dD$ form precisely \emph{two} equivalence classes.
		
	\begin{Corollary}
		There exists an inner function $\Theta$ such that 
		\begin{enumerate}\addtolength{\itemsep}{0.5\baselineskip}
			\item $\mathcal{B}_{\Theta}^{a} \cong \mathcal{B}_{\Theta}^{a'}$ for all $a, a' \in \dD \setminus \{1\}$,
			\item $\mathcal{B}_{\Theta}^{1} \not \cong \mathcal{B}_{\Theta}^{a}$ for all $a \in \dD \setminus \{1\}$.
		\end{enumerate}
	\end{Corollary}

	\begin{proof}
		This is a simple consequence of Theorem \ref{Clark} and the fact that there exists
		an inner function $\Theta$ such that $\mu_1$ is discrete but $\mu_a$ is continuous singular for every 
		$a \in \dD \setminus \{1\}$ \cite{Donoghue, MR2198367}.
	\end{proof}
	
	Provided that $a, a' \in \dD$, Theorem \ref{Clark} provides a complete characterization of when two Sedlock algebras
	$\B_{\Theta}^a$ and $\B_{\Theta}^{a'}$ are spatially isomorphic.  In this setting,
	a straightforward, measure-theoretic answer is to be expected since $\B_{\Theta}^a$ and $\B_{\Theta}^{a'}$ 
	are both algebras of normal operators.   On the other hand, if $a, a' \in \D$ then the situation turns out to be quite different.
	
	\begin{Theorem} \label{t-main-Sedlock}
		If $\Theta$ is an inner function and $a, a' \in \D$, then $\mathcal{B}_{\Theta}^{a} \cong \mathcal{B}_{\Theta}^{a'}$
		if and only if there is a unimodular constant $\zeta$ and a disk automorphism $\psi$ such that
		\begin{equation*}
			\Theta = b_{-a}(\zeta b_{a'}) \circ \Theta \circ \psi,
		\end{equation*}
		where $b_c$, for $c \in \D$, denotes the disk automorphism \eqref{eq-Factor}.
	\end{Theorem}
	
	\begin{proof}
		$(\Leftarrow)$
		We first require the following two elementary identities:
		\begin{align} 
			b_{a} \circ b_c 
			&= \left(\tfrac{1 + a \overline{c}}{1 + \overline{a} c}\right) b_{\frac{a + c}{1 + a \overline{c}}}, \qquad a, c \in \D,\label{babc}\\[5pt]
			b_{a}(\zeta z) &= \zeta b_{a \overline{\zeta}}(z), \qquad\qquad a \in \D, \zeta \in \dD. \label{bazeta}
		\end{align}
		If $\Theta = b_{-a}(\zeta b_{a'}) \circ \Theta \circ \psi$, then $\Theta_a = \zeta \Theta_{a'} \circ \psi$ whence
		\begin{equation*}
			\K_{\Theta_{a}} = \K_{\zeta \Theta_{a'} \circ \psi} = \K_{\Theta_{a'} \circ \psi}.
		\end{equation*}
		By Proposition \ref{B-psi} the unitary operator
		\begin{equation*}
			U_{\psi}: \K_{\Theta_{a'}} \to \K_{\Theta_{a'} \circ \psi} = \K_{\Theta_{a}}, \quad U f := \sqrt{\psi'} (f \circ \psi)
		\end{equation*}
		induces a spatial isomorphism between $\B_{\Theta_{a}}^{0}$ and $\B_{\Theta_{a'}}^{0}$. 
		In light of \eqref{H/aH} we have $\B_{\Theta_{a}}^{0} \cong \B_{\Theta}^{a}$ and $\B_{\Theta_{a'}}^{0} \cong \B_{\Theta}^{a'}$ 
		from which we conclude that $\B_{\Theta}^{a} \cong \B_{\Theta}^{a'}$.\medskip
		
		\noindent$(\Rightarrow)$
		Conversely suppose that $\B_{\Theta}^{a} \cong \B_{\Theta}^{a'}$. Appealing to \eqref{H/aH} once more we see that
		$\B_{\Theta_{a}}^{0} \cong \B_{\Theta_{a'}}^{0}$.  Thus there exists a unitary operator 
		$U: \K_{\Theta_{a}} \to \K_{\Theta_{a'}}$ such that $\Lambda( \B_{\Theta_{a}}^{0})  = \B_{\Theta_{a'}}^{0}$, 
		where $\Lambda(A) = U A U^{*}$.   Taking conjugates and using the fact that 
		$(\B_{\Theta_{a}}^{0})^{*} = \B_{\Theta_{a}}^{\infty}$ we obtain
		$\Lambda (\B_{\Theta_{a}}^{\infty} ) = \B_{\Theta_{a'}}^{\infty}$. In particular, this implies that
		\begin{equation} \label{BBB}
			\Lambda(\B_{\Theta_{a}}^{0} + \B_{\Theta_{a}}^{\infty})  
			= \B_{\Theta_{a'}}^{0} + \B_{\Theta_{a'}}^{\infty}.
		\end{equation}
		
		We now remark that for any inner function $u$, the weak closure of 
		$\B_{u}^{0} + \B_{u}^{\infty}$ contains $\{A_{\phi}^{u}: \phi \in L^{\infty}\}$. 
		Indeed, it is clear from the definitions of $\B_{u}^{0}$ and $\B_{u}^{\infty}$ that
		\begin{equation*}
			\B_{u}^{0} + \B_{u}^{\infty} = \{A_{\phi}^{u}: \phi \in H^{\infty} + \overline{H^{\infty}}\}.
		\end{equation*}
		By approximating $\phi \in L^{\infty}$ weak-$*$ by its Cesaro means \cite[p.~20]{Hoffman}, we see that
		$L^{\infty}$ equals the weak-$*$ closure of $H^{\infty} + \overline{H^{\infty}}$.  Therefore
		the weak closure of $\B_{u}^{0} + \B_{u}^{\infty}$ contains $\{A_{\phi}^{u}: \phi \in L^{\infty}\}$ which is 
		dense in $\mathcal{T}_{u}$ (Proposition \ref{density-Bess}). 
				Based upon the discussion in the previous paragraph and \eqref{BBB}, we conclude that
		\begin{equation*}
			\Lambda(\mathcal{T}_{\Theta_a}) = \mathcal{T}_{\Theta_{a'}}.
		\end{equation*}
	
		Theorem \ref{most-three-SI} now implies that $\Lambda$ is a product of at 
		most three spatial isomorphisms from the families
		$\Lambda_a, \Lambda_{\#}, \Lambda_{\phi}$ such that no two are of the same type. 
		Next observe that
		\begin{enumerate}\addtolength{\itemsep}{0.5\baselineskip}
			\item From \eqref{form-Jphi} we see that $\Lambda_{\psi}$ preserves analytic truncated Toeplitz operators,
			\item From \eqref{form-Jsharp} we see that $\Lambda_{\#}$ takes analytic truncated Toeplitz operators to co-analytic ones,
			\item The Crofoot transforms $\Lambda_a$ preserve neither analytic nor co-analytic truncated Toeplitz operators.
		\end{enumerate}
		Since $\Lambda(\B^{0}_{\Theta_a}) = \B^{0}_{\Theta_{a'}}$, it follows that $\Lambda = \Lambda_{\psi}$. Thus
		\begin{equation*}
			\B^{0}_{\Theta_{a'}} = \Lambda(\B^{0}_{\Theta_a}) = \B^{0}_{\zeta \Theta_a \circ \psi}.
		\end{equation*}
		Note that we must allow for the possibility of a unimodular constant $\zeta$ since the corresponding 
		Sedlock algebra does not change.  Thus  $\Theta_{a'} = \zeta \Theta_a \circ \psi$, as claimed.
	\end{proof}

	Using Theorem \ref{t-main-Sedlock} along with \eqref{sharp} yields the following corollary.

	\begin{Corollary} \label{c-ext-Sedlock-main}
		If $\Theta$ is an inner function and $a, a' \in \widehat{\C} \setminus \D^{-}$, then $\B_{\Theta}^{a} \cong \B_{\Theta}^{a'}$
		if and only if there is a unimodular constant $\zeta$ and a disk automorphism $\psi$ such that
		\begin{equation} \label{xx-ww}
			\Theta^{\#} = b_{-1/a}(\zeta b_{1/a'}) \circ \Theta^{\#} \circ \psi.
		\end{equation}
		If $a \in \D$ and $a' \in \widehat{\C} \setminus \D^{-}$, \eqref{xx-ww} is replaced by  
		\begin{equation*}
		\Theta = b_{-a} (\zeta b_{1/a'}) \circ \Theta^{\#} \circ \psi.
		\end{equation*}
	\end{Corollary}
	
\begin{Remark} \label{r-never-s-i}
We have examined when $\mathcal{B}_{\Theta}^{a} \cong \mathcal{B}_{\Theta}^{a'}$ in the case $a, a' \in \dD$ (Theorem \ref{Clark}), the case $a, a' \in \D$ (Theorem \ref{t-main-Sedlock}), the case $a, a' \in \widehat{\C} \setminus \D^{-}$, and the case $a \in \D, a' \in \widehat{\C} \setminus \D^{-}$ (Corollary \ref{c-ext-Sedlock-main}). The reader might be wondering when $\mathcal{B}_{\Theta}^{a} \cong \mathcal{B}_{\Theta}^{a'}$ in the case where $a \in \dD, a' \in \widehat{\C} \setminus \dD$. Recall from Remark \ref{normal-ops} that when $a \in \dD$, $\mathcal{B}_{\Theta}^{a}$ is an algebra of normal operators while $\mathcal{B}_{\Theta}^{a'}$, for $a \in \widehat{\C} \setminus \dD$, contains no normal operators (other than scalar multiplies of the identity). So in this situation, $\mathcal{B}_{\Theta}^{a}$, $a \in \dD$, is never spatially isomorphic to $\mathcal{B}_{\Theta}^{a'}$, $a' \in \widehat{\C} \setminus \dD$. 
\end{Remark}

	Corollary \ref{c-ext-Sedlock-main} says that when $a = 0$ and $a'  = \infty$ we have $\mathcal{B}_{\Theta}^0 \cong \mathcal{B}_{\Theta}^{\infty}$ if and only if  $\Theta = \zeta \Theta^{\#}(\psi)$. We now describe a situation when this occurs.

\begin{Corollary} \label{same-arg}
Suppose $\Theta$ is a Blaschke product whose zeros all have the same argument. Then $\mathcal{B}_{\Theta}^{0} \cong \mathcal{B}_{\Theta}^{\infty}$. 
\end{Corollary}

\begin{proof}
Since the zeros of $\Theta$ have the same argument, there is a unimodular $v$ so that the zeros of $\Theta(v z)$ are real. This means that the Blaschke products $\Theta(v z)$ and $\Theta^{\#}(\overline{v} z)$ have the same zeros and so $\Theta(v z) = \zeta \Theta^{\#}(\overline{v} z)$ for some unimodular $\zeta$. Thus $\Theta(z) = \zeta \Theta^{\#}(\overline{v}^2 z)$.  The result now follows from Corollary \ref{c-ext-Sedlock-main}.
\end{proof}

\subsection{Toeplitz matrices}	
	For specific inner functions $\Theta$, one can obtain more precise results.  For instance, if $\Theta = z^n$
	we can prove the following.

	\begin{Corollary} \label{BP-zn}
		For $a \in \D$ and $n \geq 2$, we have $\mathcal{B}_{z^n}^{a} \cong \mathcal{B}_{z^{n}}^{a'}$ if and only if $|a| = |a'|$.
	\end{Corollary}
	
	\begin{proof}
		The implication $(\Leftarrow)$ follows immediately from the identity
		\begin{equation*}
			z^n = b_{-a}(\overline{\zeta} b_{\zeta a}) \circ z^n \circ (\zeta^{1/n} z).
		\end{equation*}
		and Theorem \ref{t-main-Sedlock}.  For the $(\Rightarrow)$ implication, we start with the following two facts.\medskip
			
		\noindent\textsc{Fact 1}: If $\phi$ and $\psi$ are disk automorphisms which satisfy
		\begin{equation} \label{phicpsi}
			\phi \circ z^n = z^n \circ \psi,
		\end{equation}
		then $\phi$ and $\psi$ are both rotations. To see this, observe that if $\psi(c) = 0$, then taking the derivative 
		of \eqref{phicpsi} and evaluating at $c$ yields
		\begin{equation*}
			0 = n \psi(c)^{n - 1} \psi'(c)= \phi'(c^n) n c^{n - 1}
		\end{equation*}
		whence $c = 0$, implying that $\psi$ is a rotation.   Evaluating both sides of \eqref{phicpsi} 
		at $c = 0$ reveals that $\phi$ is also a rotation.\medskip
		
		\noindent\textsc{Fact 2}: If $a, c \in \D$ and $b_a \circ b_c$ is a rotation, then $a = -c$. To see this use \eqref{babc}.\medskip
	
		With these two facts in hand, we are ready to complete the proof. Suppose that $a, a' \in \D$ and 
		$\B_{\Theta}^{a} \cong \B_{\Theta}^{a'}$.  By Theorem \ref{t-main-Sedlock} and Fact 1, there exist unimodular $u,w$ such that
		\begin{equation*}
			B(z) = b_{-a} \circ w b_{a'} \circ B(uz),
		\end{equation*}
		where $B(z) = z^n$. Now use the fact that $B(u z) = u^n B(z)$ along with \eqref{phicpsi} to see that
		\begin{equation*}
			B = b_{-a} \circ w u^n b_{a'\overline{u}^n} \circ B = w u^n b_{-a \overline{w}\overline{u}^n} \circ b_{a'\overline{u}^{n}} \circ B.
		\end{equation*}
		By Fact 1, the automorphism pre-composing $B$ is a rotation.  Fact 2 now implies that 		
		$a \overline{w} \overline{u}^n = a'\overline{u}^n$ and hence $a' = a \overline{w}$.  In particular, this implies that $|a| = |a'|$.
	\end{proof}

	\begin{Corollary}
		Suppose that $a, a' \in \C \cup \{\infty\}$.
		\begin{enumerate}\addtolength{\itemsep}{0.5\baselineskip}
			\item If $a, a' \in \D$, then $\B_{z^n}^{a} \cong \B_{z^n}^{a'} \Leftrightarrow |a| = |a'|$.
			\item If $a, a' \in \widehat{\C} \setminus \D^{-}$, then $\B_{z^n}^{a} \cong \B_{z^n}^{a'} \Leftrightarrow |a| = |a'|$.
			\item If $0 < |a| < 1$ and $|a'| > 1$, then $\B_{z^n}^{a} \cong \B_{z^n}^{a'} \Leftrightarrow |a a'| = 1$.
			\item If $a, a' \in \dD$, then $\B_{z^n}^{a} \cong \B_{z^n}^{a'}$.
			\item $\B_{z^n}^{0} \cong \B_{z^n}^{\infty}$.
		\end{enumerate}
	\end{Corollary}
	
	\begin{proof}
		Use the previous several results along with \eqref{sharp}.
	\end{proof}

\subsection{The atomic inner function}

	The opposite extreme to Corollary \ref{BP-zn} occurs with the singular atomic inner function.
	
	\begin{Theorem} \label{t-atomic-basic}
		If $\Theta$ denotes the atomic inner function
		\begin{equation} \label{atomic-basic}
			\Theta(z) = \exp\left(-\frac{1 + z}{1 - z}\right),
		\end{equation}
		then, for $a, a' \in \D$, we have $\B_{\Theta}^{a} \cong \B_{\Theta}^{a'} \Leftrightarrow a = a'$.
	\end{Theorem}

\begin{proof}
	We first note that if $|\zeta| = 1$, then by \eqref{babc} and \eqref{bazeta} we get
	\begin{equation} \label{b-comp}
		b_{-a}(\zeta b_{a'})(z) = \frac{\zeta - a \overline{a'}}{1 - \overline{a} a' \zeta} 
		\left(\frac{z - (\frac{\zeta a' - a}{\zeta - a \overline{a'}})}{1 - (\frac{\overline{a'} - \overline{a} \zeta}{1 - \overline{a} a'\zeta}) z}\right).
	\end{equation}
	If $\B_{\Theta}^{a} \cong \B_{\Theta}^{a'}$, then  by Theorem \ref{t-main-Sedlock} there exists a 
	$\zeta \in \dD$ and an automorphism $\psi$ such that 
	\begin{equation} \label{bb-comp}
		\Theta = b_{-a}(\zeta b_{a'}) \circ \Theta \circ \psi.
	\end{equation}
	We will first argue that $a = \zeta a'$.  If this were not the case, then by \eqref{b-comp} the map 
	$b_{-a}(\zeta b_{a'}) \circ \Theta \circ \psi$ will have a zero in $\D$ (since $\Theta \circ \psi$ maps $\D$
	onto $\D \setminus \{0\}$) which cannot happen by \eqref{bb-comp} and because $\Theta$ has no zeros in $\D$.

	Having shown that $a = \zeta a'$, we now claim that $\zeta = 1$. 
	To do this we observe by using \eqref{b-comp} and \eqref{bb-comp} again that $\Theta = \zeta (\Theta \circ \psi)$.
	Writing
	\begin{equation*}
	\psi(z) = \lambda \frac{z-a}{1 - \overline{a}z}
	\end{equation*}
	we find
	\begin{equation*}
		\frac{ \Theta(z) }{ \Theta( \psi(z) ) } = \exp\left( - \frac{1+z}{1-z} + \frac{1 + \psi(z) }{1 - \psi(z) } \right).
	\end{equation*}
	A little algebra reveals that
	\begin{equation*}
		- \frac{1+z}{1-z} + \frac{1 + \psi(z) }{1 - \psi(z) }
		=2 \left( \frac{ z^2 \overline{a} + z(\lambda - 1) - a \lambda}{ (z-1)(z (\lambda + \overline{a}) - a \lambda - 1) } \right),
	\end{equation*}
	which is constant precisely when $a=0$ and $\lambda = 1$.  In other words, $\psi(z) = z$ and $\zeta = 1$,
	from which we conclude that $a = a'$.
\end{proof}

Using Theorem \ref{Clark}, and Remarks \ref{Clark-remark} and \ref{Clark-needed} we get the following.

\begin{Corollary}
If $\Theta$ is the atomic inner function \eqref{atomic-basic}, then
$\B_{\Theta}^{a} \cong \B_{\Theta}^{a'}$ whenever $a, a' \in \dD$.
\end{Corollary}

From the proof of Theorem \ref{t-atomic-basic} we see the following.

\begin{Corollary}
If $\Theta$ is any singular inner function and $a, a' \in \D$, then $\mathcal{B}_{\Theta}^{a} \cong \mathcal{B}_{\Theta}^{a'} \Rightarrow |a| = |a'|$.
\end{Corollary}

This next group of results shows that when there is some sort of symmetry in the inner function $\Theta$, we can have spatially isomorphic Sedlock algebras. We will make this more precise in Theorem  \ref{equalityoftheas} below. For now we begin with a few examples. 

\begin{Proposition} \label{uv}
Suppose that $\Theta$ is inner such that there is a $u \in \dD \setminus \{1\}$ with $\Theta(u z) = v \Theta(z)$ for some $v \in \dD \setminus \{1\}$. Then for any $a \in \D$, $\mathcal{B}_{\Theta}^{a} \cong \mathcal{B}_{\Theta}^{a \overline{v}}$.
\end{Proposition}

\begin{proof}
With $\phi(z) = v z$ and $\psi(z) = u v$, a simple computation shows that $\Theta = \phi \circ \Theta \circ \psi$. Using \eqref{bazeta} we see that $\phi(z) = b_{-a} (v b_{a
\overline{v}})$. Now use Theorem \ref{t-main-Sedlock}.
\end{proof}

Proposition \ref{uv} will be generalized in Lemma \ref{bmazba} below. 

\begin{Example}
\begin{enumerate}
\item If $\Theta$ is any odd inner function, then $\mathcal{B}_{\Theta}^{a} \cong \mathcal{B}_{\Theta}^{-a}$ for any $a \in \D$. One can see this by letting $u = v = -1$ in Proposition \ref{uv}.
\vskip .05in
\item Fix $z_0 \in \D \setminus\{0\}$ and $n \in \mathbb{N}$. Let
$$\Theta(z) = z b_{a_1}(z) b_{a_2}(z) \cdots b_{a_n}(z),$$
where $a_1, a_2, \ldots, a_n$ are the $n$-th roots of $z_0$.
If $u$ is a primitive root of unity one can check that
\begin{equation*}
\Theta(u^k z) = u^k \Theta(z) \tag{$1 \leq k \leq n - 1$}
\end{equation*} and so for any $a \in \D$ we have $\mathcal{B}_{\Theta}^{a} \cong \mathcal{B}_{\Theta}^{a \overline{u}^k}$.
\vskip .05in
\item Let $\Theta(z) = z S_{\mu}(z)$, where $S_{\mu}$ is the singular inner function with singular measure $\mu = \delta_1 + \delta_{-1} + \delta_{i} + \delta_{-i}$. A computation shows that $S_{\mu}(i z) = S_{\mu}(z)$ and so $\Theta(i z) = i \Theta(z)$. This with $u = v = i$ in Proposition \ref{uv} we see that
$\mathcal{B}_{\Theta}^{a} \cong \mathcal{B}_{\Theta}^{-i a}$ for any $a \in \D$. One can continue this as follows: If $u$ is a primitive $n$th root of unity and $\mu$ has unit point masses at $u^k$, $k = 1, \ldots,  n$, then $S_{\mu}(u^k z) =  S_{\mu}(z)$.  From here we have $\Theta(uz) = u \Theta(z)$.  Then for each $a  \in \D$, 
$\mathcal{B}^{a}_{\Theta} \cong \mathcal{B}^{u^k}_{\Theta}$ for $k = 1, 2, \ldots, n$.
\end{enumerate}
\end{Example}

We have seen examples where $\mathcal{B}_{\Theta}^{a} \cong \mathcal{B}_{\Theta}^{a'}$ with $a \not = a'$ and some examples where $\mathcal{B}_{\Theta}^{a} \cong \mathcal{B}_{\Theta}^{a'}$ implies $a = a'$. What are conditions on $\Theta$ so that $\mathcal{B}_{\Theta}^{a} \cong \mathcal{B}_{\Theta}^{a'}$ always implies $a = a'$?

\begin{Theorem} \label{equalityoftheas}
For an inner function $\Theta$, the following are equivalent. 
\begin{enumerate}
\item If  $a, a' \in \widehat{\C} \setminus \dD$ and $\mathcal{B}_{\Theta}^{a} \cong \mathcal{B}_{\Theta}^{a'}$, then $a = a'$. 
\vskip .05in
\item If $\phi, \psi$ are disk automorphisms with either $\phi \circ \Theta = \Theta \circ \psi$ or $\phi \circ \Theta = \Theta^{\#} \circ \psi$ then $\phi(z) = z$. 
\end{enumerate}
\end{Theorem}

The proof of Theorem \ref{equalityoftheas} requires the following technical lemma. 

\begin{Lemma} \label{bmazba}
Let $\psi$ be a disk automorphism. Then for each $a \in \D$, there is a $\zeta \in \dD$ and $a' \in \D$ so that 
$\psi = b_{-a}(\zeta b_{a'})$. 
\end{Lemma}

\begin{proof}
Let 
\begin{equation*}
\psi(z) = \lambda b_c. \tag{$\lambda \in \dD, c \in \D$}
\end{equation*}
Note, for $a, a' \in \D$ and $\zeta \in \dD$, that
$$b_{-a}(\zeta b_{a'}) = \lambda b_c \Leftrightarrow b_a(\lambda b_c) = \zeta b_{a'}.$$
From \eqref{b-comp} we see that 
\begin{equation} \label{mud}
\zeta = \lambda \frac{1 + a \overline{\lambda} \overline{c}}{1 + \overline{a} \lambda c}, \quad a' = \frac{a \overline{\lambda} + c}{1 + \overline{c} a \overline{\lambda}}.
\end{equation} 
This completes the proof.
\end{proof}

\begin{proof}[Proof of Theorem \ref{equalityoftheas}] Without loss of generality, we will assume that $a, a' \in \D$. Assume (ii) and suppose that $\mathcal{B}_{\Theta}^{a} \cong \mathcal{B}_{\Theta}^{a'}$. By Theorem \ref{t-main-Sedlock} we know there is a $\zeta \in \dD$ and a disk automorphism $\psi$ so that 
$$b_{-a} (\zeta b_{a'}) \circ \Theta = \Theta \circ \psi.$$ But by our assumption (ii) we see that 
$b_{-a} (\zeta b_{a'})$ is the identity automorphism. From \eqref{mud} it follows that $a = a'$, which proves (i). 

Conversely suppose that (i) holds and assume that $\phi, \psi$ are disk automorphisms with $\phi \circ \Theta   = \Theta \circ \psi$. Our goal is to show that $\phi(z) = z$. In Lemma \ref{bmazba} choose $a = 0$ to produce $\zeta \in \dD$ and $a' \in \D$ so that $\phi = b_{-0}(\zeta b_{a'})$. By Theorem \ref{t-main-Sedlock}  we have 
$\mathcal{B}_{\Theta}^{0} \cong \mathcal{B}_{\Theta}^{a'}$ and so, by our assumption (i), it must be the case that $a' = 0$. Thus $\phi(z) = \zeta z$. We will now show that $\zeta = 1$.

Choose $a \not = 0$ and argue from above that 
$\phi = b_{-a}(\zeta_a b_a)$ for some $\zeta_a \in \dD$. But from \eqref{mud} we have 
$$b_{-a} (\zeta_a b_a) = \mu b_d,$$
where 
$$\mu = \frac{\zeta_a - |a|^2}{1 - |a|^2 \zeta_a}, \quad d = \frac{\zeta_a a - a}{\zeta_a - |a|^2}.$$ But
$\phi(z) = \zeta z$ and so $d = 0$ (which implies $\zeta_a = 1$ and $\mu = 1$) and $\mu = \zeta$. Thus $\zeta = 1$. This proves (ii). Our proof is now complete. 
\end{proof}

Theorem \ref{equalityoftheas} has an interesting corollary. 

\begin{Corollary} \label{cor-auto}
Suppose $a, a' \in \widehat{\C}\setminus \dD$ with $a \not = a'$, and $\mathcal{B}_{\Theta}^{a} \cong \mathcal{B}_{\Theta}^{a'}$.
\begin{enumerate}
\item If $a, a' \in \D$, then there is a non-trivial automorphism $\psi$ of $\widehat{\C}$ mapping $\D$ to itself so that $\mathcal{B}_{\Theta}^{c} \cong \mathcal{B}_{\Theta}^{\psi(c)}$ for every $c \in \D$.
\vskip .05in
\item If  $a, a' \in \widehat{\C} \setminus \D^{-}$, then there is a non-trivial automorphism $\psi$ of $\widehat{\C}$ mapping $\widehat{\C} \setminus \D^{-}$ to itself so that $\mathcal{B}_{\Theta}^{c} \cong \mathcal{B}_{\Theta}^{\psi(c)}$ for every $c \in \widehat{\C} \setminus \D^{-}$.
\vskip .05in
\item If $a \in \D, a' \in \widehat{\C} \setminus \D^{-}$, then  there is an automorphism $\psi$ of $\widehat{\C}$ mapping $\D$ to $\widehat{\C} \setminus \D^{-}$ so that $\mathcal{B}_{\Theta}^{c} \cong \mathcal{B}_{\Theta}^{\psi(c)}$ for every $c \in \D$.
\end{enumerate}
\end{Corollary}
\begin{proof}

Proof of (i): From \eqref{mud} we see that 
$$b_{-a}(\zeta b_{a'}) = \mu b_{d},$$
where 
$$\mu = \frac{\zeta - a \overline{a'}}{1 - \overline{a} a' \zeta}, \quad d = \frac{\zeta a' - a}{\zeta - a \overline{a'}}.$$ From Lemma \ref{bmazba} we know that for each $c \in \D$, there is a $w \in \dD$ and a $c' \in \D$ so that 
$$b_{-a}(\zeta b_{a'}) = b_{-c} (w b_{c'}).$$
By Theorem  \ref{t-main-Sedlock}  (applied to $\mathcal{B}_{\Theta}^{a} \cong \mathcal{B}_{\Theta}^{a'}$ and $\mathcal{B}_{\Theta}^{c} \cong \mathcal{B}_{\Theta}^{c'}$) we conclude that $\mathcal{B}_{\Theta}^{c} \cong \mathcal{B}_{\Theta}^{c'}$. Note, from \eqref{mud} that $$ c' = \frac{c + \mu d}{\mu + c \overline{d}}.$$
If we define 
$$\psi(c) = \overline{\mu} \frac{c + d \mu}{1 + c \overline{\mu d}}$$ then $\psi$ is a disk automorphism with the desired properties. 
\vskip .05in 

Proof of (ii): By Corollary \ref{c-ext-Sedlock-main}, there is a (non-trivial) disk automorphism $\psi$ so that 
$\mathcal{B}_{\Theta^{\#}}^{c} \cong \mathcal{B}_{\Theta^{\#}}^{\psi(c)}$ for $c \in \D$. By Proposition \ref{sharp-B} we have $\mathcal{B}_{\Theta}^{1/c} \cong \mathcal{B}_{\Theta}^{1/\psi(c)}$. 
\vskip .05in

Proof (iii): By  By Corollary \ref{c-ext-Sedlock-main}, there is a (non-trivial) disk automorphism $\psi$ so that $\mathcal{B}_{\Theta}^{c} \cong \mathcal{B}_{\Theta^{\#}}^{\psi(c)}$. Now apply Proposition \ref{sharp-B} to get 
$\mathcal{B}_{\Theta^{\#}}^{\psi(c)} \cong \mathcal{B}_{\Theta}^{1/\psi(c)}.$
\end{proof}

\begin{Example}
From Corollary \ref{same-arg} we know that if $\Theta$ is a Blaschke product whose zeros all have the same argument then $\mathcal{B}_{\Theta}^{0} \cong \mathcal{B}_{\Theta}^{\infty}$. From the techniques in the proof of Corollary \ref{cor-auto} we see that there is a $\zeta \in \dD$ such that $\mathcal{B}_{\Theta}^{c} \cong \mathcal{B}_{\Theta}^{\zeta/c}$ for every $c \in \D$.
\end{Example} 


\vskip .05in

The proof of Theorem \ref{t-main-Sedlock} can be easily modified to prove the following.

\begin{Theorem} \label{main-IJ}
Suppose $\Theta_1, \Theta_2$ are inner functions and $a_1, a_2 \in \D$. Then $$\mathcal{B}_{\Theta_1}^{a_1} \cong \mathcal{B}_{\Theta_2}^{a_2}$$
if and only if there is a unimodular constant $\zeta$ and a disk automorphism $\psi$ such that
\begin{equation} \label{IJ-ab-in-D}
\Theta_1 = b_{-a_1}(\zeta b_{a_2}) \circ \Theta_2 \circ \psi.
\end{equation}
If $a_1, a_2 \in \widehat{\C} \setminus \D^{-}$, then condition \eqref{IJ-ab-in-D} is replaced by 
\begin{equation*} 
\Theta_1^{\#} = b_{-1/a_1}(\zeta b_{1/a_2}) \circ (\Theta_2)^{\#} \circ \psi.
\end{equation*}
If $a_1 \in \D$ while $a_2 \in \widehat{\C} \setminus \D^{-}$ is in the exterior disk, then the condition \eqref{IJ-ab-in-D} is replaced by 
\begin{equation*} 
\Theta_1 = b_{-a_1}(\zeta b_{1/a_2}) \circ (\Theta_2)^{\#} \circ \psi.
\end{equation*}
\end{Theorem}

\begin{Remark}
It is worth mentioning again (see Remark \ref{r-never-s-i}) that $\mathcal{B}_{\Theta_1}^{a_1}$, $a_1 \in \dD$, is never spatially isomorphic to $\mathcal{B}_{\Theta_2}^{a_2}$, $a_2 \in \widehat{\C} \setminus \dD$. 
\end{Remark}



\section{Isometric isomorphisms and Pick algebras}
	
	To conclude this paper, we consider the closely related question of whether or not isometric
	isomorphisms of Sedlock algebras are necessarily spatially implemented.
	To be more specific, suppose, for two inner functions $\Theta_1$ and $\Theta_2$ and extended complex numbers $a_1, a_2 \in \widehat{\C}$,
that $\B_{\Theta_1}^{a_1}$ is isometrically isomorphic to $\B_{\Theta_2}^{a_2}$.
	Is it necessarily the case that $\B_{\Theta_1}^{a_1}$ is spatially isomorphic to $\B_{\Theta_2}^{a_2}$?
	In certain cases, the answer is yes.

	\begin{Theorem} \label{main-II-SI}
		If $\Theta_1$ and $\Theta_2$ are finite Blaschke products with $n$ distinct zeros and $a_1, a_2 \in \widehat{\C}$,  then the algebras 
		$\B_{\Theta_1}^{a_1}$ and $\B_{\Theta_2}^{a_2}$ are isometrically isomorphic if and only if they are spatially isomorphic. 
	\end{Theorem}
	
	The proof of Theorem \ref{main-II-SI} requires a few preliminaries.
	Fix $n$ \emph{distinct} points $z_1, z_2, \ldots, z_n$ in $\D$ and consider 
	the following inner product on $\C^{n}$: For vectors
	$$\vec{u} = (u_1, u_2, \ldots, u_n), \quad \vec{v} = (v_1, v_2, \ldots, v_n),$$
	in $\C^n$ define
	\begin{equation} \label{IP}
		(\vec{u}, \vec{v})_{\vec{z}} := \sum_{j, k = 1}^{n} \frac{u_j \overline{v_k}}{1 -  z_j \overline{z_k}},
	\end{equation}
	where $\vec{z} = (z_1, z_2, \ldots, z_n)$.  To emphasize the fact that $\C^n$ has been endowed with this 
	inner product, we use the notation $\C^{n}_{\vec{z}}$. 

	For a fixed vector $\vec{w} = (w_1, w_2,\ldots, w_n)$ we define the corresponding diagonal operator 
	$R_{\vec{w}}: \C^{n}_{\vec{z}} \to \C^{n}_{\vec{z}}$ by setting, for $\vec{u} = (u_1, u_2,\ldots, u_n)$,
	\begin{equation*}
		R_{\vec{w}} (\vec{u}) = (u_1 w_1,u_2 w_2, \ldots, u_n w_n).
	\end{equation*}
	  Among other things, it is clear that
	$$R_{\vec{w_1}} R_{\vec{w_2}} = R_{\vec{w_1} \bullet \vec{w_2}}$$ where $\vec{w_1} \bullet \vec{w_2}$ denotes the
	entrywise product of $\vec{w_1}$ and $\vec{w_2}$.  This implies that the set
	\begin{equation*}
		\mathfrak{U}_{{\vec{z}}}:= \{R_{\vec{w}}: \vec{w} \in \C^n\}
	\end{equation*}
	forms an algebra of operators on $\C^{n}_{\vec{z}}$. This algebra, studied by B.~Cole, K.~Lewis, and J.~Wermer \cite{CWL, CW}, 
	is called the \emph{Pick algebra}.

	\begin{Lemma} \label{Sed-Pick}
		If $\Theta$ is a $n$-fold Blaschke product with distinct zeros $\vec{z} = (z_1,z_2, \ldots, z_n)$, then 
		$\B^{\infty}_{\Theta} \cong \mathfrak{U}_{\vec{z}}$.
	\end{Lemma}
	
	\begin{proof}
		It is well-known that the reproducing kernels 
		\begin{equation*}
			k_{z_j}(z) := \frac{1}{1 - \overline{z_j} z}, \tag{$1 \leq j \leq n$}
		\end{equation*}
		from \eqref{eq-ReproducingKernel} form a (non-orthogonal) basis for the model space $\K_{\Theta}$. 
		Define the \emph{unitary} operator $U: \K_{\Theta} \to \C^{n}_{\vec{z}}$ by setting
		\begin{equation*}
			U\left(\sum_{j = 1}^{n} a_j k_{z_j} \right)= (a_1, a_2,\ldots, a_n).
		\end{equation*}
		The fact that $U$ is unitary comes from the fact that $\C^{n}_{\vec{z}}$ is equipped with the 
		inner product in \eqref{IP}.  Since	
		\begin{equation*}
			A_{\overline{\phi}} k_{z_j} = \overline{\phi(z_j)} k_{z_j}
		\end{equation*}
		holds for $\phi$ in $H^{\infty}$, we have
		\begin{align*}
			U A_{\overline{\phi}} \left(\sum_{j = 1}^{n} a_j k_{z_j} \right) 
			& = (\overline{\phi(z_1)} a_1,\overline{\phi(z_2)} a_2, \ldots, \overline{\phi(z_n)} a_n)\\
			&  = R_{\vec{w}} (a_1, a_2,\ldots, a_n)\\
			& = R_{\vec{w}} U \left(\sum_{j = 1}^{n} a_j k_{z_j} \right),
		\end{align*}
		where $\vec{w} = (\overline{\phi(z_1)}, \overline{\phi(z_2)},\ldots, \overline{\phi(z_n)})$. 
		Now use interpolation to show that 
		\begin{equation*}
			U \B_{\Theta}^{\infty} U^{*} = \mathfrak{U}_{\vec{z}}.
		\end{equation*}
		Hence $\B_{\Theta}^{\infty} \cong \mathfrak{U}_{\vec{z}}$. 
	\end{proof}

	The proof of Theorem \ref{main-II-SI} requires one more little detail. For fixed $a \in \D$, let $w_1,w_2, \ldots, w_n$ be distinct points in $\D$ which satisfy $\Theta(w_j) = a$.
	As Sedlock demonstrated, the operators
	\begin{equation*}
		Q_j := \frac{1}{\Theta'(w_j)} C k_{w_j} \otimes k_{w_j},  \tag{$j = 1, 2,\ldots, n$}
	\end{equation*}
	belong to $\B_{\Theta}^{a}$.  Moreover, it is not hard to show that the $Q_j$
	are idempotents which form a non-orthogonal resolution of the identity:
	\begin{equation*}
		Q_{j}^2 = Q_j, \quad 
		\sum_{j = 1}^{n} Q_j = I, \quad 
		Q_j Q_l = \delta_{j, l} Q_j, \quad 
		\T_{\Theta} = \bigvee_{j=1}^n\{Q_j, Q_{j}^{*}\}.
	\end{equation*}
	Since $Q_{j}^{*} \in \mathcal{B}_{\Theta}^{1/\overline{a}}$ we see that 
	\begin{equation*}
		\B_{\Theta}^{a} = \bigvee_{j=1}^n \{Q_{j}\}.
	\end{equation*}
	Furthermore, since each $Q_j$ is a non-selfadjoint idempotent we also have
	\begin{equation*}
		\|Q_{j}\| > 1 . \tag{$j = 1, \ldots, n$}
	\end{equation*}

	The setup for the case $a \in \dD$ is handled in a similar manner.  Indeed, if $a \in \dD$, let 
	$\zeta_1,\zeta_2, \ldots, \zeta_n$ be the distinct (necessarily unimodular) solutions to 
	the equation $\Theta(\zeta_j) = a$.  As before, Sedlock shows that the orthogonal projections
	\begin{equation*}
		P_{j} = \frac{1}{\sqrt{\Theta'(\zeta_j)}} k_{\zeta_j} \otimes k_{\zeta_j}, \tag{$j=1,2,\ldots,n$}
	\end{equation*}
	belong to $\B_{\Theta}^a$.  Moreover, we also observe that the $P_j$ form a resolution of the identity
	\begin{equation*}
		P_{j}^2 = P_j, \quad 
		\sum_{j = 1}^{n} P_j = I, \quad 
		P_j P_l = \delta_{j, l} P_j, \quad 
		\T_{\Theta} = \bigvee_{j=1}^n \{P_j, P_{j}^{*}\},
	\end{equation*}
	and that
	\begin{equation*}
		\B_{\Theta}^{a} = \bigvee_{j=1}^n \{P_{j}\}.
	\end{equation*}
	Furthermore, each $P_j$ is an orthogonal projection whence $\norm{P_j} = 1$.

We are now ready to finish off the proof of Theorem \ref{main-II-SI}. 

\begin{proof}[Proof of Theorem \ref{main-II-SI}]: For a finite Blaschke product $\Theta$ with\emph{ distinct} zeros and $a \in \D$ we have 
\begin{align*}\hspace{1.5in}
\mathcal{B}_{\Theta}^{a} & \cong \mathcal{B}_{\Theta_a}^{0} && \mbox{(by \eqref{H/aH}) }\\
& \cong \mathcal{B}_{(\Theta_a)^{\#}}^{\infty} && \mbox{(by \eqref{sharp})}\\
& \cong \mathfrak{U}_{\vec{z}} && \mbox{(by Proposition \ref{Sed-Pick})}
\end{align*}
where $\vec{z}$ is the vector of distinct zeros of $(\Theta_{a})^{\#}$. For $a \in \widehat{\C} \setminus \D^{-}$, 
\begin{align*}\hspace{1.65in}
\mathcal{B}_{\Theta}^{a} & \cong \mathcal{B}_{\Theta^{\#}}^{1/a} && \mbox{(by \eqref{sharp})}\\
& \cong \mathcal{B}_{(\Theta^{\#})_{1/a}}^{0} && \mbox{(by \eqref{H/aH})}\\
& \cong \mathcal{B}_{((\Theta^{\#})_{1/a})^{\#}}^{\infty} && \mbox{(by \eqref{sharp})}\\
& \cong \mathfrak{U}_{\vec{z}}, && \mbox{(by Proposition \ref{Sed-Pick})}
\end{align*}
where $\vec{z}$ is the vector of distinct zeros of $((\Theta^{\#})_{1/a})^{\#}$. 

Now suppose that $a_1, a_2 \in \widehat{\C} \setminus \dD$ with $\mathcal{B}_{\Theta_1}^{a_1}$ and $\mathcal{B}_{\Theta_2}^{a_2}$ isometrically isomorphic. Then, by the computation above, their corresponding Pick algebras are isometrically isomorphic. 
However, two Pick algebras are isometrically isomorphic if and only if they are spatially isomorphic \cite{CW},
whence, by the above computations, $\B_{\Theta_1}^{a_1} \cong \B_{\Theta_2}^{a_2}$.

If $a_1, a_2 \in \dD$, then, by Corollary \ref{nothingtoprove}, $\mathcal{B}_{\Theta_1}^{a_1} \cong \mathcal{B}_{\Theta_2}^{a_2}$ and so there is nothing to prove. 

If $a_1 \in \dD$ and $a_2 \in \widehat{\C} \setminus \dD$ we see, using the above discussion,  that any isometric isomorphism will map $Q_j$ to $P_{\sigma(j)}$, for some permutation $\sigma$ of $\{1, 2, \ldots, n\}$. But since $\|P_{\sigma(j)}\| = 1$ and $\|Q_{j}\| > 1$, we see that this case never arises. The proof is now complete. 
\end{proof}

An interesting application to this theorem is the following Corollary.

\begin{Corollary} \label{quotient}
Suppose that $\Theta_1$ and $\Theta_2$ are finite Blaschke products with $n$ distinct zeros. Then the quotient algebras $H^{\infty}/\Theta_1 H^{\infty}$ and $H^{\infty}/\Theta_2 H^{\infty}$ are isometrically isomorphic if and only if there is a unimodular constant $\zeta$ and a disk automorphism $\psi$ so that $\Theta_1= \zeta \Theta_2 \circ \psi$.
\end{Corollary}

\begin{proof}
By means of extremal problems \cite{NLEPHS} or Hankel operators \cite{MR2597679} one can show,
for any inner function $\Theta$ and $\phi \in H^{\infty}$, that $$\|A_{\phi}\| = \mbox{dist}(\phi/\Theta, H^{\infty}).$$ This means that $\mathcal{B}_{\Theta}^{0}$ is isometrically isomorphic to $H^{\infty}/\Theta H^{\infty}$. The corollary now follows from Theorem \ref{main-II-SI} and Theorem \ref{main-IJ} . 
 \end{proof}

\bibliography{SIATTO}

\end{document}